\documentclass[oneside,11pt]{amsart}

\usepackage[utf8]{inputenc}
\usepackage[T1]{fontenc} 
\usepackage{microtype}
\usepackage[english]{babel}
\usepackage{pdfpages}
\usepackage{emptypage}
\usepackage{geometry}
\usepackage[lined]{algorithm2e}

\usepackage{graphicx}

\usepackage{amsfonts}
\usepackage{amssymb,amsthm,amsmath,amstext,amsxtra}
\usepackage{bm}       
\usepackage{mathtools} 
\usepackage{hyperref}
\definecolor{mylinkcolor}{rgb}{0.5,0.0,0.0}
\definecolor{myurlcolor}{rgb}{0.0,0.0,0.75}
\hypersetup{colorlinks=true,urlcolor=myurlcolor,citecolor=myurlcolor,linkcolor=mylinkcolor,linktoc=page,breaklinks=true}
\usepackage{colonequals}
\usepackage{float} 
\usepackage{color}
\usepackage{xcolor}
\usepackage{xspace}
\usepackage{numprint}

\usepackage{chngcntr}
\counterwithin{figure}{section}

\usepackage[normalem]{ulem}

\usepackage{csquotes} 
\usepackage[style=ieee-alphabetic,backend=biber, sorting=nyt, doi=false,isbn=false,url=false]{biblatex}
\newbibmacro{string+doiurlisbn}[1]{%
  \iffieldundef{doi}{%
    \iffieldundef{url}{%
      \iffieldundef{isbn}{%
        \iffieldundef{issn}{%
          #1%
        }{%
          \href{http://books.google.com/books?vid=ISSN\thefield{issn}}{#1}%
        }%
      }{%
        \href{http://books.google.com/books?vid=ISBN\thefield{isbn}}{#1}%
      }%
    }{%
      \href{\thefield{url}}{#1}%
    }%
  }{%
    \href{http://dx.doi.org/\thefield{doi}}{#1}%
  }%
}

\DeclareFieldFormat{title}{\usebibmacro{string+doiurlisbn}{\mkbibemph{#1}}}
\DeclareFieldFormat[article,incollection,inproceedings]{title}%
    {\usebibmacro{string+doiurlisbn}{\mkbibquote{#1}}}

\usepackage{tikz-cd}
\usepackage{chngcntr}

\counterwithin*{equation}{section}
\counterwithin*{equation}{subsection}

\usepackage{yfonts}
\usepackage{mathrsfs}
\usepackage{bbm}
\usepackage[%
  cal=euler,
  scr=rsfso%
]{mathalfa} 

\usetikzlibrary{decorations.pathmorphing}

\DeclareMathOperator{\DC}{DC}
\DeclareMathOperator{\Alb}{Alb}

\DeclareMathOperator{\End}{End}

\DeclareMathOperator{\im}{im}

\DeclareMathOperator{\spec}{Spec}

\DeclareMathOperator{\gal}{Gal}

\DeclareMathOperator{\Frob}{Frob}

\DeclareMathOperator{\rk}{rk}

\DeclareMathOperator{\chr}{char}

\newcommand{\dd}{\mathrm{d}}

\DeclareMathOperator{\vol}{vol}

\newcommand{\dR}{{\mathrm{dR}}}
\newcommand{\betti}{{\mathrm{Betti}}}

\newcommand{\crys}{{\mathrm{crys}}}
\newcommand{\integral}{{\mathrm{int}}}

\DeclareMathOperator{\NS}{NS}

\DeclareMathOperator{\Span}{span}

\DeclareMathOperator{\Gr}{Gr}

\newcommand{\cxhat}{\widehat\cx}
\newcommand{\Zxp}{\Zz_{\xp}}
\newcommand{\Qxp}{\Qq_{\xp}}

\newcommand{\Qxpbar}{\overline{\Qq}_{\xp}}

\newcommand{\Qp}{\Qq_p}
\newcommand{\Qbar}{\overline{\Qq}}

\newcommand{\Hprim}{H_{\mathrm{prim}}}
\newcommand{\Hlog}{H_{\mathrm{log}}}
\newcommand{\Hu}{H_{U}}
\newcommand{\Ttilde}{\widetilde T}

\usepackage{array}   
\newcolumntype{M}{>{$}c<{$}} 
\newcolumntype{L}{>{$}l<{$}}
\newcolumntype{R}{>{$}r<{$}}
\usepackage{arydshln}
\usepackage{multirow}

\newcommand{\frakp}{{\mathfrak{p}}}

\DeclareMathOperator{\Hom}{\mathscr{H}\kern -2pt om}
\DeclareMathOperator{\Ext}{\mathscr{E}\kern -1.5pt xt}

\newtheorem{theorem}{Theorem}[section]
\numberwithin{theorem}{section} 

\newcommand{\nameofthm}{}

\newtheorem{genericThm}[theorem]{\nameofthm}

{%
  \begin{proof}$ $\newline
  \indent ($\implies$) 
}
{\end{proof}}

\newtheorem*{theorem*}{Theorem}

\newtheorem{lemma}[theorem]{Lemma}
\newtheorem{corollary}[theorem]{Corollary}
\newtheorem{proposition}[theorem]{Proposition}

\newtheorem{conjecture}[theorem]{Conjecture}

\theoremstyle{definition}

\newtheorem*{definition*}{Definition}
\newtheorem{definition}[theorem]{Definition}

\newtheorem{remark}[theorem]{Remark}
\newtheorem{example}[theorem]{Example}

\newtheorem*{notation*}{Notation}

\newcommand{\nameofenv}{}
\newtheorem*{generic}{\nameofenv}

\newcommand{\nameOfNumEnv}{}
\newtheorem{genericNumbered}[theorem]{\nameOfNumEnv}

\makeatletter
\def\imod#1{\allowbreak\mkern10mu({\operator@font mod}\,\,#1)}
\makeatother

\newcommand{\del}{\partial}

\newcommand{\ppp}{\Pp^3}

\providecommand{\binom}[2]{{#1\choose#2}}

\newcommand{\too}{\longrightarrow}
\newcommand{\tos}{\twoheadrightarrow}
\newcommand{\toi}{\hookrightarrow}
\newcommand{\isoto}{\overset{\sim}{\to}}

\newcommand{\Cc}{\mathbbm{C}}

\newcommand{\Ff}{\mathbbm{F}}

\newcommand{\Nn}{\mathbbm{N}}
\newcommand{\Qq}{\mathbbm{Q}}
\newcommand{\Pp}{\mathbbm{P}}

\newcommand{\Zz}{\mathbbm{Z}}

\renewcommand{\H}{\mathrm{H}}
\newcommand{\F}{\mathrm{F}}

\newcommand{\CH}{\mathrm{CH}}

\newcommand{\xp}{\mathfrak{p}}
\newcommand{\xP}{\mathfrak{P}}
\newcommand{\xq}{\mathfrak{q}}

\newcommand{\cc}{\mathcal{C}}

\newcommand{\cj}{\mathcal{J}}

\newcommand{\co}{\mathcal{O}}

\newcommand{\cp}{\mathcal{P}}

\newcommand{\cu}{\mathcal{U}}

\newcommand{\cv}{\mathcal{V}}
\newcommand{\cX}{\mathcal{X}}
\newcommand{\cx}{\mathcal{X}}
\newcommand{\cy}{\mathcal{Y}}
\newcommand{\cz}{\mathcal{Z}}

\renewcommand{\k}{{\sf k}}

  \title[Effective obstruction to Tate classes]{Effective obstruction to lifting Tate classes from positive characteristic}
\date{\today}
\author[E. Costa]{Edgar Costa}
\address{Dept. of Mathematics, Massachusetts Institute of Technology,
Cambridge, MA, USA}
\email{edgarc@mit.edu}
\urladdr{\url{https://edgarcosta.org}}

\author[E. C. Sertöz]{Emre Can Sertöz}
\address{Max Planck Institute for Mathematics in the Sciences, 04103 Leipzig, Germany}
\email{emresertoz@gmail.com}
\urladdr{\url{https://emresertoz.com}}

\newcounter{ouralgo}[section]
\renewcommand{\theouralgo}{\arabic{section}.\arabic{ouralgo}}

\begin{document}
\begin{abstract}
  We give an algorithm that takes a smooth hypersurface over a number field and computes a $p$-adic approximation of the obstruction map on the Tate classes of a finite reduction.
  This gives an upper bound on the ``middle Picard number'' of the hypersurface.
  The improvement over existing methods is that our method relies only on a single prime reduction and gives the possibility of cutting down on the dimension of Tate classes by two or more.
  The obstruction map comes from $p$-adic variational Hodge conjecture and we rely on the recent advancement by Bloch--Esnault--Kerz to interpret our bounds.
\end{abstract}

\maketitle

\section{Introduction}
Let $X_{\Cc} \colonequals Z(f) \subset \Pp^{2r+1}_\Cc$ be a smooth complex hypersurface  of even dimension $2r$ with the polynomial $f$ having algebraic coefficients.
The dimension $\rho^r(X_{\Cc})$ of the space of codimension $r$ algebraic cohomology classes $A^r(X_{\Cc}) \otimes_{\Zz} \Qq$ in $\H_{\betti}^{2r}(X_\Cc,\Qq)$ is called the \emph{middle Picard number} of $X_\Cc$.
Due to the Lefschetz hyperplane theorem, there are no interesting algebraic cycles in codimension different from $r$.
The middle Picard number is a difficult invariant to compute, but it strongly constrains the geometry, arithmetic, and dynamics of the hypersurface.
Recently, there have been numerous developments on computing the middle Picard number and we review them in Section~\ref{sec:literature}.

The main contributions of this paper are Algorithm~\ref{alg:main} and Proposition~\ref{prop:even_better_bound}.
They give practical means for computing upper bounds for $\rho^r(X_{\Cc})$ starting with the defining equation~$f$. The key idea is to take a prime reduction of the hypersurface and compute the obstruction to lifting algebraic (or Tate) classes of the reduction to characteristic zero.
We give an implementation\footnote{The code is available at \url{https://github.com/edgarcosta/crystalline_obstruction}.} of this algorithm in \texttt{SageMath}~\cite{sage}.
In Section~\ref{sec:examples}, this implementation is used to illustrate the strengths and weaknesses of the method.
To the best of our knowledge, our method gives a superior performance over any other method that has a computer implementation.

There is a practical method, due to van~Luijk~\cite{vanluijk-07}, for bounding the Picard number of K3 surfaces, which has now become the standard.
When the K3 surface $X$ is defined over a number field, the Picard numbers of its prime reductions give upper bounds on the Picard number $\rho^1(X_\Cc)$.
The method of van~Luijk can improve these bounds by at most one, via a comparison of two different reductions.
However, some K3 surfaces with real multiplication do not admit a prime reduction where the Picard number jump is less than two~\cite{charles-14}.

Our approach is different, and we do not suffer from these limitations.
For instance, we do not need two distinct prime reductions to improve the bound. We also see in Example~\ref{ex:real_multiplication} that our method does not necessarily break down in the presence of real multiplication.
Additionally, in Example~\ref{ex:galois}, we give an example of a K3 surface, whose Picard number jumps by six at a prime reduction, and show that our method detects this jump.

Let us recall that for simply connected surfaces such as K3s, the Picard group and its image in integral cohomology, namely the Néron--Severi group, coincide.
These are free abelian groups whose rank equals the Picard number.

In Example~\ref{ex:database}, we go through the \numprint{184725} smooth quartic K3 surfaces in the database of~\cite{lairez-sertoz} to verify computations and benchmark performance. In Section~\ref{sec:examples}, we give many more examples, involving Picard numbers of quintic surfaces and endomorphisms of Jacobian threefolds.

For the rest of the paper we will work with number fields as opposed to the field of complex numbers. In the end we do not lose generality since $\rho^r(X_{\Cc}) = \rho^r(X_{\Qbar})$, see for instance Lemma~\ref{lem:rhor_over_CC_and_Qbar}.

\subsection{Lifting algebraic cycles from finite characteristic}

We will now sketch the core concept behind our algorithm. To simplify the exposition, let us assume $f$ has rational coefficients so that $X = Z(f) \subset \Pp^{2r+1}_{\Qq}$. For most prime numbers $p \in \Zz$ the reduction $X_p \colonequals Z(f \bmod{p}) \subset \Pp^{2r+1}_{\Ff_p}$ will be well defined and smooth. Fix one such prime $p$. The choice of such a prime introduces an infinite sequence of hypersurfaces
\begin{equation}
  Z(f \bmod{p^n}) \subset \Pp^{2r+1}_{\Zz/p^n}
\end{equation}
that can be said to approximate $X$.

As we shall soon see, one of the strong points of working with a finite field $\Ff_p$ is that the determination of algebraic cycles on the variety $X_p /\Ff_p$ is greatly simplified in comparison to $X/\Qq$. After determining algebraic cycles on $X_p$, one must determine which cycles can be lifted to $Z(f \bmod{p^n})$ for all $n$, and eventually to $X$.

Hodge theory gives an obstruction to lifting cycles from characteristic $p$, namely the map~\eqref{eq:intro_pi} below. The $p$-adic variational Hodge conjecture states that any unobstructed cycle lifts to $X$. See~\cite{emerton-padic-variation} and~\cite{AMMN} for more details on this conjecture.
Bloch--Esnault--Kerz~\cite{bloch-esnault-kerz-14} prove that an intermediate statement is true: an unobstructed cycle can lift compatibly to all the hypersurfaces $Z(f \bmod{p^n})$ approximating~$X$. Except when $r=1$, it is unknown whether this is sufficient for lifting to $X$.

As an example, suppose $X$ is a K3 surface. Then the obstruction to lifting a curve $C \subset X_p$ can be represented by a $p$-adic integer $a_1p +a_2p^2+\dots \in p\Zz_p$. The curve $C$ is the restriction of a curve in $Z(f \bmod{p^n})$ if and only if $a_1=a_2=\dots=a_{n-1}=0$, see also~\cite{elsenhans-jahnel-11b}.

The crystalline cohomology is the cohomology theory we need in order to compute obstructions of cycles. This theory associates a well behaved cohomology to smooth and proper varieties over finite fields~\cite{berthelot-ogus,berthelot74,CL98}. A remarkable feature of the theory is that the crystalline cohomology of $X_p$ can be identified with the algebraic de~Rham cohomology of $X$ with $p$-adic coefficients (Section~\ref{sec:crystalline_to_dR}),
\begin{equation}
 \H^{2r}_\crys (X_p / \Zz_p) \otimes \Qp \simeq \H_{\dR}^{2r}(X/\Qq) \otimes_{\Qq} \Qp.
\end{equation}

The algebraic de~Rham cohomology is defined as the hypercohomology of the usual differential complex $\Omega_{X/\Qq}^{\bullet}$, see~\cite{grothendieck-dR} for the conception of this idea and~\cite[\S 3]{huber17} for more details.
The resulting cohomology groups have coefficients in the base field and when tensored with real or complex numbers, one recovers the classical de~Rham cohomology groups.

The image of the crystalline cycle class map is a group $A^r(X_p) \subset \H^{2r}_\crys (X_p / \Zz_p)$ representing algebraic cohomology classes of $X_p$.
On the other hand, the cycle class map for de~Rham cohomology also gives us the group $A^r(X) \subset \H_{\dR}^{2r}(X/\Qq)$ of algebraic cycles on $X$ that are defined over $\Qq$.
Let us denote the $\Qq$-spans of $A^r(X)$ and $A^r(X_p)$ with $A^r(X)_{\Qq}$ and $A^r(X_p)_{\Qq}$ respectively.

There is a containment $\iota \colon A^r(X)_{\Qq} \toi A^r(X_p)_{\Qq}$ obtained by reducing subvarieties of $X$ modulo $p$, see Section~\ref{sec:cycle_class_maps}.
Moreover, comparing with complex Hodge theory one sees that $A^r(X)$ must be contained in the $r$-th piece of the Hodge filtration $\F^r  \H^{2r}_\crys (X_p / \Zz_p)$ on $\H^{2r}_\crys (X_p / \Zz_p)$.
Consider the quotient map
\begin{equation}\label{eq:intro_pi}
\pi \colon  \H^{2r}_\crys (X_p / \Zz_p) \to  \H^{2r}_\crys (X_p / \Zz_p) / \F^r  \H^{2r}_\crys (X_p / \Zz_p)
\end{equation}
and note that $\pi(A^r(X))= 0$. The theorem of Bloch--Esnault--Kerz (Theorem~\ref{thm:simplified bloch-esnault-kerz-14}) states that any element in the kernel of
\begin{equation}
  \pi_A \colonequals \pi|_{A^r(X_p)}
\end{equation}
lifts to a compatible sequence of algebraic cycles on $Z(f \bmod{p^n})$ for each $n \ge 1$. It is unknown whether every element in $\ker \pi_A$ comes from an element in $A^r(X)$, see \emph{loc.\ cit.} However, when $r=1$ we do have $A^1(X)=\ker \pi_A$ due to~\cite[Theorem~5.1.4]{EGAIII-2}.

See Section~\ref{sec:obstruction_map} for more details on the obstruction map $\pi_A$.
We also recommend the introduction of the paper~\cite{bloch-esnault-kerz-14} for the intuition it provides in interpreting their main result.

\subsection{Tate classes as a substitute for algebraic classes}\label{sec:intro_tate}

Tate's conjecture gives a hold on the space of algebraic cohomology classes $A^r(X_p)$. The Frobenius action on $X_p$ induces an action on crystalline cohomology by functoriality and, therefore, a map $\Frob \colon \H^{2r}_\crys (X_p / \Zz_p) \to \H^{2r}_\crys (X_p / \Zz_p)$. Tate's conjecture states that the $\Qp$-span of $A^r(X_p)$ coincides with $p^r$-eigenspace of $\Frob$.

Write $T^r(X_p) \colonequals \ker \left( \Frob -p^r \right)$.
The elements of $T^r(X_p)$ are called Tate classes. The inclusion $A^r(X_p) \subset T^r(X_p)$ is always true and $A^r(X_p)_{\Qq} \otimes_{\Qq} \Qp \simeq T^r(X_p)$ is the Tate conjecture. This conjecture is known to hold for K3 surfaces~\cite{charles-13, madapusi-13, kim-madapusi-16}.

Consider the restriction of the quotient map~\eqref{eq:intro_pi} to the space of Tate classes,
\begin{equation}
  \pi_T \colonequals \pi|_{T^r(X_p)}.
\end{equation}
By the discussion above, $A^r(X)$ is contained in $\ker \pi_T$ and one can deduce the inequality
\begin{equation}\label{eq:intro_bound}
  \dim_{\Qq} A^r(X)_{\Qq} \le \dim_{\Qp} \ker \pi_T.
\end{equation}

Our Algorithm~\ref{alg:main} describes a practical method to bound $\dim_{\Qp} \ker \pi_T$. We should point out that the codomain of $\pi_T$ is typically small when viewed as a $\Qp$-vector space but infinite dimensional as a $\Qq$-vector space. If $X$ is a K3 surface then the codomain of $\pi_T$ is $1$ dimensional. Therefore, as it stands, the bound \eqref{eq:intro_bound} can be underwhelming. Ideally, one should seek to recover a $\Qq$-structure on $T^r(X_p)$ containing $A^r(X_p)$ to leverage the fact that $\dim_{\Qq} \Qp = \infty$. In this way, the codomain of $\pi_T$ would have the capacity to obstruct arbitrarily many classes. We do not take this approach in this paper, but we say more about it in Section~\ref{sec:limitation}.

Instead, we make use of the Galois module structure on Tate classes to improve~\eqref{eq:intro_bound}.
A slight modification of the discussion above allows one to pass to the algebraic closures of the base fields. This allows us to compute bounds on $\dim_{\Qq} A^r(X_{\Qbar}) = \dim_{\Qq} A^r(X_{\Cc})$.
Furthermore, these bounds on $\dim_{\Qq} A^r(X_{\Cc})$ can be computed without leaving the comfort of $\H^{2r}_\crys (X_p / \Zz_p)$.
We explain how to enlarge the base field in Section~\ref{sec:enlarging base field}. In the process, we show that algebraic cycles appearing over different base fields can be obstructed independently. This is the content of Proposition~\ref{prop:even_better_bound}, which improves over \eqref{eq:intro_bound}.

\subsection{A note on using finite approximations}

A $p$-adic approximation of the Frobenius action on crystalline cohomology $\Frob \colon \H^{2r}_\crys (X_p / \Zz_p) \to \H^{2r}_\crys (X_p / \Zz_p)$ can be obtained as in~\cite{abbott-kedlaya-roe-10, costa-phd, chk}. Using this approximation, we find a vector space approximating the space of Tate cycles $T^r(X_p)$.
We work in a coordinate system that respects the Hodge structure, see Section~\ref{sec:griffiths}, and allows us to approximate the obstruction map $\pi_T$ on $T^r(X_p)$.
With careful management of precision, we can compute a rigorous upper bound on the dimension of the kernel of $\pi_T$, see Section~\ref{sec:dim Li}.

\subsection{A limitation and the need for integral structure}\label{sec:limitation}

There is a shortcoming of our approach. Even when the Tate conjecture holds, and even when $A^r(X)$ is characterized in $A^r(X_p)$ by the kernel of $\pi_A$, the inequality $\dim_\Qq A^r(X)_{\Qq} \le \dim_{\Qp} \ker \pi_T$ can be strict. This is because $A^r(X)_{\Qq}$ is a rational vector space, while the obstruction map can be irrational, see Example~\ref{ex:weak} where Proposition~\ref{prop:even_better_bound} cannot help.

This issue can be circumvented to a large extent if one can identify the image of $A^r(X_p)$ inside $T^r(X_p)$.  It is a well-known shortcoming of the Tate conjecture that, unlike the Hodge conjecture, there is no description of this image, even conjecturally.

Suppose we can compute (or approximate) any $\Zz$-lattice $\Lambda$ inside $T^r(X_p)$ that would contain $A^r(X_p)$. Using pLLL~\cite{pLLL}, the $p$-adic version of LLL~\cite{LLL}, we can get a good guess on the ``integral'' kernel of $\pi_A$ restricted to $\Lambda$. This would sacrifice rigor in return for what is most probably the correct value of $\rho^r(X_\Cc)$.

This idea can be compared to the one used in~\cite{lairez-sertoz} in the setting of complex periods. There, integral Betti cohomology serves as the $\Zz$-lattice $\Lambda$ inside the complex Betti cohomology.

In the present paper, we do not consider the problem of identifying a lattice $\Lambda \subset T^r(X_p)$. Nevertheless, we tried to set-up the theory in a way that anticipates this development. We addressed issues of torsion in integral crystalline cohomology and the state of knowledge regarding the properties of the obstruction map with integral coefficients, see Sections~\ref{sec:torsion_free_obstruction} and~\ref{sec:griffiths}.

\subsection{Previous approaches}\label{sec:literature}

Given the importance of $A^r(X)$, several techniques exist in the literature for obtaining information on $A^r(X)$ or $\rho^r(X)$ for a given $X$.

For example, the authors of~\cite{poonen-testa-vanluijk-15} provide an algorithm for surfaces conditional on the computability of the \'{e}tale cohomology with finite coefficients.
In~\cite{hassett-kresch-tschikel-13}, the authors provide an algorithm for K3 surfaces of degree 2 conditional on an effective version of the Kuga--Satake construction.
Another algorithm to compute the geometric Picard number of a K3 surface, conditional on the Hodge conjecture for $X\times X$, is presented in~\cite{charles-14}.

Often, as in the algorithms mentioned above, one obtains a lower bound for $\rho(X)$ by exhibiting divisors explicitly.
However, there is no known practical algorithm to do this in general.
Nonetheless, if $X$ has some additional structure, practical methods may arise.
For example, when $X$ is a product of curves~\cite{cmsv},
is a quotient of another variety by finite group~\cite{shioda-86},
or
is an elliptic surface~\cite{shioda-72, shioda-90}.

The specialization homomorphism of algebraic cycles $\iota \colon A^1(X) \toi A^1(X_p)$
is used frequently to compute Picard numbers of surfaces.
For instance, one may compare the lattice structure on $A^1(X_p)$ for two different primes to limit the image of~$\iota$~\cite{vanluijk-07}, or use the Artin--Tate conjecture for surfaces~\cite{kloosterman}.
One can also view $\iota$ as a morphism of two Galois modules, as was done in~\cite{elsenhans-jahnel-11a}, see also Proposition~\ref{prop:even_better_bound}.
Alternatively, when explicit elements of $A^1(X_p)$ are known, one can rely on their geometry to show that some of them cannot be lifted. This becomes a powerful tool in the absence of torsion in the cokernel of $\iota$, see~\cite{elsenhans-jahnel-11b}.

The methods outlined in the previous paragraph are strongest when the rank jump between $A^1(X)$ and $A^1(X_p)$ is at most one. However, Charles~\cite{charles-14} proved that some K3 surfaces may never admit a prime reduction where $\dim A^1(X_p) - \dim A^1(X) \le 1$, see Example~\ref{ex:real_multiplication}.

We should highlight the difficulty in computing Picard numbers of surfaces. For instance, only recently did Schütt~\cite{schutt-quintic} obtain the set of values that can be attained as the Picard number of a quintic surface. We still do not know this set for sextic surfaces. See the introduction and \S~2 of \emph{loc.\ cit.}\ for a comprehensive overview.

The papers~\cite{sertoz18, lairez-sertoz} tackle the computation of the middle Picard number of a hypersurface using complex transcendental methods. That method seems very reliable (Example~\ref{ex:database}), but proving the result of its computation can be challenging, see for instance~\cite{movasati-sertoz}.

\subsection{Overview}

In Section~\ref{sec:lifting_cycles}, we present the theoretical framework of our method. This includes an overview of the problem of lifting algebraic cycles from positive characteristic. We state the Bloch--Esnault--Kerz theorem (Theorem~\ref{thm:simplified bloch-esnault-kerz-14}) and show how to use it in computations, to constrain the liftability of certain Tate classes.
In Section~\ref{sec:computing_in_crystalline}, we recall how to effectively compute in the crystalline cohomology of a smooth hypersurface. Here, we slightly strengthen the known results to better handle torsion.
In Section~\ref{sec:algorithm}, we present the main algorithm of this paper (Algorithm~\ref{alg:main}) and clarify each computational step.
In Section~\ref{sec:examples}, we explore examples that illustrate the strengths and weaknesses of the method. We also demonstrate the usage of the code \verb|crystalline_obstruction|.

\subsection*{Acknowledgements}
We thank Bjorn Poonen for his help in initiating the project and for his guidance throughout.
We also benefited greatly from the constant support of Kiran Kedlaya.
We thank Matthias Schütt and John Voight for their valuable comments on the first version of this text.
We are grateful for the careful comments from the anonymous referees and from Avi Kulkarni to the earlier version of this article.
The first author was supported by the Simons Collaboration in Arithmetic Geometry, Number Theory, and Computation via Simons Foundation grant~550033.

\section{Lifting algebraic cycles from positive characteristic}\label{sec:lifting_cycles}

In this section, we will be working with a smooth projective variety $X$ defined over a number field $K$. The goal is to constrain the dimension of the span of algebraic cohomology classes on $X_{\overline{K}} \colonequals X\times_K \overline{K}$, for an algebraic closure $\overline{K}$ of $K$.

We begin with a review of the core concepts we will use. In Section~\ref{sec:passage_to_finite_fields} we setup notation for the passage to a finite field from $K$. We then recall the statement of Tate's conjecture for finite fields in Section~\ref{sec:tate}.

The later subsections are intended to reposition these concepts to simplify computations. In particular, we do not want to extend the base for crystalline cohomology when computing. Section~\ref{sec:enlarging base field} makes the first simplification that allows us to find ``eventual Tate classes'', see~\eqref{eq:Ttilde}. In Section~\ref{sec:improve_obstruction}, we show that the obstruction map can be studied at the level of eventual Tate classes and the obstruction map can be applied seperately on elements that appear at different levels of field extensions. The conclusion, Proposition~\ref{prop:even_better_bound}, multiplies the extent to which we can obstruct classes.

\subsection{Passage from number fields to finite fields}\label{sec:passage_to_finite_fields}

Let $K$ be a number field with ring of integers $\co_K$.
Fix an unramified prime ideal $\xp \subset \co_K$. Localizing and completing $\co_K$ at $\xp$ we get a local field $\Qxp$ and its ring of integers $\Zxp$. We will continue to write $\xp$ for the maximal ideal of $\Zxp$. Let $\k$ be the residue field at $\xp$ and $p$ the characteristic of $\k$.

The ring of Witt vectors $W(\k)$ of $\k$ and $\Zxp$ are canonically isomorphic. In fact, for some $n$ there is an isomorphism $\k \simeq \Ff_{p^n}$ and $W(k)$ is determined uniquely up to isomorphism as the unramified extension of degree $n$ of the ring of $p$-adic integers $\Zz_p$.

\subsubsection{Good reduction.} With $X \to \spec K$ a smooth projective variety, assume that $\xp\subset \co_K$ is chosen so that that $X$ has \emph{good reduction} at $\xp$. That is, we assume that there is a regular scheme $\cx \to \spec \Zxp$, smooth over the base, such that $X_{\Qxp} \colonequals X \times_K \Qxp$ is identified with the generic fiber $\cx \times_{\Zxp} \Qxp$.  We will write $X_{\xp}$ for the special fiber of $\cx$, and $\overline{X}_{\xp} \colonequals X_{\xp} \times_k \overline{k}$ for its base change to an algebraic closure $\overline{k}$ of $k$.

We will also use ``thickenings'' of $X_\xp$, namely $\cx_n \colonequals \cx \otimes_{\Zxp} (\Zxp/\xp^n)$ for each $n\ge 1$. Let $\widehat \cx$ be the formal scheme obtained by completing $\cx$ along the special fiber $X_\xp$.

\subsubsection{The specialization map on subvarieties} For any $r\ge 0$ and a reduced scheme $Y$, the algebraic cycle group $\cz^r(Y)$ of $Y$ is the free abelian group generated by codimension $r$ subvarieties of $Y$. When $Y$ is a formal scheme, we use $\cz^r(Y)$ to denote formal subschemes of codimension $r$. Let us write $\cz^r(\cx)' \subset \cz^r(\cx)$ to be the free abelian group generated by subvarieties in $\cx$ that are flat over the base $\Zz_\xp$.

There is an isomorphism $\cz^r(X_{\Qxp}) \to \cz^r(\cx)'$ which maps each subvariety $V \subset X$ to its closure in $\cx$. The inverse of this map is given by intersecting a flat variety in $\cx$ with the generic fiber $X_{\Qxp}$. There is also a restriction map $\cz^r(\cx)' \to \cz^r(X_\xp)$ obtained by intersecting a subvariety of $\cx$ by the special fiber $X_\xp$. Composing the isomorphism above with the restriction map gives the \emph{specialization map}~\cite[\S 20.3]{fulton--intersection_theory}
\begin{equation}
  \mathrm{sp} \colon \cz^r(X_{\Qxp}) \to \cz^r(X_\xp).
\end{equation}

Let us point out that the specialization map factors through $\cz^r(\widehat\cx)$,
\begin{equation}\label{eq:factor_sp}
  \cz^r(X_{\Qxp}) \to \cz^r(\widehat \cx) \to \cz^r(X_\xp),
\end{equation}
though the first map does not need to be surjective.

\subsubsection{A Hodge filtration on crystalline cohomology}
For each $r\ge 0$, the de~Rham cohomology $\H^r_{\dR}(\cx/\Zxp)$ of $\cx$ comes with the Hodge filtration $\F^\bullet \H^r_{\dR}(\cx/\Zxp)$.
Using the Berthelot comparison isomorphism~\cite{berthelot-97, shiho-02} (see also Section~\ref{sec:crystalline_to_dR}),
\begin{equation}\label{eq:berthelot_iso_simple}
  \Phi \colon \H^r_{\dR}(\cx/\Zxp) \isoto \H^r_{\crys}(X_\xp/\Zxp),
\end{equation}
we can carry the Hodge filtration over to the crystalline cohomology of $X_\xp$. Let us point out that the resulting filtration is not intrinsic to $X_\xp$ but depends on the model $\cx/\Zxp$. The filtration modulo $\xp$ is, however, intrinsic to~$X_\xp$.

\begin{definition}
  \label{def:induced_filtration}
  Let $\F_{\cx}^\bullet \H^r_{\crys}(X_\xp/\Zxp)$ denote the filtration induced by the Hodge filtration on de~Rham cohomology carried over by the comparison isomorphism~\eqref{eq:berthelot_iso_simple}.
\end{definition}

\subsubsection{Cycle class maps}\label{sec:cycle_class_maps}

There are cycle class maps $c_{\dR}$~\cite{hartshorne--deRham} and $c_{\crys}$~\cite{gillet-messing--cycle_classes} (see also~\cite[\S 8]{bloch-esnault-kerz-14} for a review of the crystalline cycle map with integral coefficients) making the following diagram commutative:

\begin{equation}\label{eq:cycle_class_maps}
  \begin{tikzcd}
    \cz^r(X) \arrow[r,"\sim"]\arrow[d,"c_{\dR}"] & \cz^r(\cx){'} \arrow[r] \arrow[d] & \cz^r(X_\xp) \arrow[d,"c_{\crys}"] \\
    \H^{2r}_{\dR}(X/\Qxp) & \H^{2r}_{\dR}(\cx/\Zxp) \arrow[l] \arrow[r,"\sim","\Phi"'] & \H^{2r}_{\crys}(X_\xp/W)
  \end{tikzcd}
\end{equation}
We define the middle vertical arrow so as to make the square on the right commute.
The map $\Phi$ on the bottom right is the Berthelot comparison map~\eqref{eq:berthelot_iso_simple}. The map on the bottom left is the composition:
\begin{equation}
  \H^{2r}_{\dR}(\cx/\Zxp) \to \H^{2r}_{\dR}(\cx/\Zxp)\otimes_{\Zxp} \Qxp \isoto \H^{2r}_{\dR}(X/\Qxp).
\end{equation}

\begin{definition}
  The image of the cycle class maps, in the appropriate cohomology, will be denoted with $A^r(X)$ and $A^r(X_\xp)$. Their tensor with $\Qq$ are denoted by $A^r(X)_{\Qq}$ and $A^r(X_\xp)_{\Qq}$, respectively.
  The elements in the image of the cycle class maps, or elements in their $\Qq$-span, will be called \emph{algebraic cohomology classes}.
\end{definition}

\begin{lemma}\label{lemma:iota}
  The maps above give an injective homomorphism $\iota \colon A^r(X)_{\Qq} \toi A^r(X_{\xp})_{\Qq}$.
\end{lemma}
\begin{proof}
  The cycle class maps are compatible in the sense that the diagram~\eqref{eq:cycle_class_maps} commutes after tensoring with $\Qq$.
  Furthermore, the horizontal bottom arrows are all isomorphisms once tensored with $\Qq$, and thus we have injectivity.
\end{proof}

\begin{remark}
  In fact, even without tensoring with $\Qq$, we can obtain an injection. The specialisation map $A^r(X_{\overline{K}}) \to A^r(\overline{X}_{\xp})$ preserves the intersection pairing, see~\cite[Corollary~20.3]{fulton--intersection_theory}. Using that the polarization maps to the polarization, and using the Hodge index theorem on $A^r(X_{\overline{K}})$, we conclude that no element can map to zero.
\end{remark}

\subsubsection{Dimensions of the space of algebraic cycles over different fields}
Let $\rho^r_{\betti}(X_\Cc)$ denote the dimension of the space of codimension $r$ algebraic cycles $A^r(X_\Cc)$ in the Betti cohomology  $\H^{2r}_{\betti}(X_\Cc,\Cc)$ of the associated complex manifold.

Let $L/K$ be a field extension of $K$. Define $\rho^r_{\dR}(X_{L})$ as the dimension of the algebraic cycles in the \emph{de~Rham} cohomology of $X_{L} \colonequals X \times_K L$.

Let $\overline{K} \subset \Cc$ be the algebraic closure of $K \toi \Cc$. Let $K \toi \Qxp$ be a localization as in Section~\ref{sec:passage_to_finite_fields} and let $\Qxpbar$ be an algebraic closure of $\Qxp$. Then we have the following series of equalities:
\begin{equation}
  \rho^r_{\betti}(X_\Cc) = \rho^r_{\dR}(X_\Cc) = \rho^r_{\dR}(X_{\overline{K}}) = \rho^r_{\dR}(X_{\Qxpbar}).
\end{equation}
The first equality follows from the standard comparison isomorphisms. The other equalities result from the following well-known fact.

\begin{lemma}\label{lem:rhor_over_CC_and_Qbar}
 Suppose $K$ is algebraically closed and characteristic zero. Let $L/K$ be a field extension with $L$ algebraically closed. Let $X_K/K$ be a smooth projective variety and $X_L$ its pullback to $L$. Then for every $r\ge 0$, $\rho^r_{\dR}(X_K) = \rho^r_{\dR}(X_L)$.
\end{lemma}
\begin{proof}[Sketch of proof]
  In characteristic zero, the cycle class map $c_{\dR}$ is well defined, and it factors through the Chow groups $\CH^r$ (the algebraic cycle group $\cz^r$ modulo algebraic equivalence). But the Chow groups of $X_K$ and $X_L$ are canonically isomorphic.

  The inclusion $\CH^r(X_K) \toi \CH^r(X_L)$ is induced by pulling back varieties in $X_K$ to $X_L$. In the reverse direction, take a subvariety $V \subset X_L$. Without loss of generality, we may assume $L$ is the field of definition of $V$ over $K$. In particular, $L$ now has finite transcendence degree over $K$. We can find an affine $K$-variety $Y$ with function field $L$ over which $V$ extends to a flat family of varieties $\cv \to Y$ so that $\cv \toi X_K \times_K Y$ is a closed immersion. Any fiber of $\cv \to Y$ over a $K$-valued point defines an element in $\CH^r(X_K)$.
\end{proof}

\subsubsection{The obstruction map.}\label{sec:obstruction_map} The goal here is to give a partial description of the image of the inclusion $\iota \colon A^r(X)_{\Qq} \toi A^r(X_{\xp})_{\Qq}$. That is the image of the composed map $\cz^r(X) \to \cz^r(X_\xp) \to \H^{2r}_{\crys}(X_\xp/\Zxp)$. The following theorem instead allows us to describe the image of the composed map $\cz^r(\cxhat) \to \cz^r(X_\xp) \to \H^{2r}_{\crys}(X_\xp/\Zxp)$ (see also~\eqref{eq:factor_sp}).

\begin{definition}
  An algebraic cycle $\xi \in A^r(X_\xp)$ is \emph{unobstructed} (with respect to $\cx$) if $\xi \in \F^r_\cx\H^{2r}_{\crys}(X_\xp/\Zxp)$, with the filtration defined in~Definition~\ref{def:induced_filtration}.
  A cycle is \emph{obstructed} otherwise.
\end{definition}

Note that unobstructed cycles form a subgroup. The terminology is justified by the following theorem of Bloch, Esnault, and Kerz, which states that modulo torsion, the group of unobstructed cycles, can be lifted to $\cxhat$.

\begin{theorem}[{\cite[Theorem~1.3]{bloch-esnault-kerz-14}}]
  \label{thm:simplified bloch-esnault-kerz-14}
  Suppose $\cx/\Zz_\xp$ is smooth and projective. For each $r\ge 0$ and $p = \chr \k > \dim X_\xp + 6$ the image of the composition $\cz^r(\cxhat)\otimes \Qq \to \cz^r(X_\xp) \otimes \Qq \tos A^r(X_\xp)_{\Qq}$ coincides with the group of unobstructed cycles.
\end{theorem}

\begin{definition}\label{def:obstruction_map}
  The \emph{obstruction map} of $X_\xp$ with respect to $\cx$ is the quotient map
  \[
    \pi_{\cx} \colon  \H_{\crys}^{2r}(X_\xp/\Zxp) \to \H_{\crys}^{2r}(X_\xp/\Zxp)/\F_\cx^r \H_{\crys}^{2r}(X_\xp/\Zxp).
  \]
  If it is clear from context, we may tensor $\pi_{\cx}$ with $\Qxp$ without changing notation. Restrictions of $\pi_{\cx}$ to a subspace $V$ of cohomology will be abbriviated to $\pi_{V}$.
\end{definition}

\begin{remark}
  If we are given $X_\xp$ but not $X$ the model $\cx/\Zxp$ is not unique.
  The choice of a model $\cx/\Zxp$ impacts which cycles are obstructed.
\end{remark}

\begin{remark}
  A theorem of this nature --- in particular, the ability to define the Chern class map for line bundles --- was the original motivation for the definition of crystalline cohomology~\cite[\S 7.4]{grothendieck--crystals}. For line bundles, this goal was realized by Berthelot and Ogus in~\cite{berthelot-ogus}.
\end{remark}

\subsection{Finding the image of the Chern class map via Tate's conjecture}\label{sec:tate}

Due to its computational complexity, we would like to avoid computing and representing actual subvarieties in $X$ or even $X_{\xp}$ as much as possible. On the other hand, the $\Qxp$-span of algebraic cycles in cohomology has, at least conjecturally, a computationally tractable description as \emph{Tate classes}~\cite{tate-bsd-66, milne-07}.

Recall that there is the arithmetic Frobenius map $X_\xp \to X_\xp$ over $\k$.
As crystalline cohomology is functorial, the relative Frobenius map induces a map on cohomology:
\begin{equation}\label{eq:frob}
  \Frob_{\k} \colon \H_{\crys}^{2r}(X_\xp/\Zxp) \to \H_{\crys}^{2r}(X_\xp/\Zxp).
\end{equation}

\begin{definition}
  The \emph{integral Tate classes} are the $q^{-r}\Frob_\k$-fixed elements of crystalline cohomology, they form $T^r(X_\xp)_{\integral} \subset \H_{\crys}^{2r}(X_\xp/\Zxp)$.
  The fixed elements of $q^{-r}\Frob_{\k}$ in $\H_{\crys}^{2r}(X_\xp/\Zxp) \otimes \Qxp$ are \emph{Tate classes} and they form the space $T^r(X_{\xp})$.
\end{definition}

\begin{conjecture}[Tate Conjecture]
  \label{conj:Tate conjecture}
  For each $r$, algebraic cycles span $T^r(X_{\xp})$, that is, $A^r(X_{\xp})\otimes_{\Zz} \Qxp \simeq T^r(X_\xp)$.
\end{conjecture}

\begin{remark}
  The stronger statement $A^r(X_{\xp})\otimes_{\Zz} \Zxp = T^r(X_\xp)_{\integral}$ is called the \emph{integral Tate conjecture}. This version is often false due to torsion, though it may be true modulo torsion. In any case, when $r=1$, then the usual Tate conjecture implies the integral Tate conjecture~\cite{milne-07}.
\end{remark}

\begin{remark}
  Tate himself proved the Tate conjecture for abelian varieties over finite fields~\cite{tate-66}. More recently, the Tate conjecture was proven for K3 surfaces~\cite{charles-14, madapusi-13, kim-madapusi-16}.
For an overview of the Tate conjecture over finite fields, as well as recent progress on it, we recommend~\cite{milne-07, totaro-17}.
\end{remark}

\subsection{Enlarging the base field}\label{sec:enlarging base field}

Given the smooth projective variety $X/K$ we want to compute the dimension of $A^r(X_{\overline{K}})$, with $\overline{K}$ an algebraic closure of $K$. If $X_\xp$ is our choice of finite reduction, it is likely that $A^r(X_{\overline{K}})$ will not map into $T^r(X_{\xp})$ and we need to enlarge the residue field $\k$. The dimension of the space of Tate classes will increase as we enlarge $\k$ but will eventually stop growing. This terminal dimension, as well as the terminal dimension of the unobstructed Tate cycles, can be computed without ever enlarging the base field. We will explain how to do this here, paying attention to computational limitations.

Let $\chi_{2r} \in \Qq[t]$ be the characteristic polynomial of $q^{-r} \Frob_{\k}$, a scaling of the Frobenius map~\eqref{eq:frob}.
Factorize this characteristic polynomial over $\Qq[t]$ as follows
\begin{equation}
\label{eqn:Qfact}
      \chi_{2r}(t) = h(t) \prod_i \Phi_{i} (t)^{\gamma_i},
\end{equation}
where, for each $i$, the polynomial $\Phi_{i}(t) \in \Zz[t]$ is the $i$-th cyclotomic polynomial, i.e., the minimal polynomial of any primitive $i$-th root of unity, the exponents $\gamma_i \ge 0$, and no root of $h(t)$ is a root of unity. Set $u \colonequals \operatorname{lcm}\{i \mid \gamma_i \neq 0\}$.

Consider now the following subspace of $\H_{\crys}^{2r}(X_\xp/\Zxp)\otimes \Qxp$,
\begin{equation}\label{eq:Ttilde}
\begin{aligned}
\Ttilde^r(X_{\xp})  \colonequals & \ker\left( \Frob_\k^u - q^{ur} \right).
\end{aligned}
\end{equation}
Elements of $\Ttilde^r(X_{\xp})$ may be viewed as ``eventual Tate classes'' as they will span the Tate classes when the base field $\k$ is enlarged, as its dimension will match the rank of $A\bigl((X_\xp)_{\overline{k}}\bigr)$, as shown below.

\begin{proposition}
  Using the obstruction map (Definition~\ref{def:obstruction_map}) we obtain the following bound:
  \begin{equation}
  \dim A^r(X_{\overline{K}})_{\Qq} \le \dim  \ker \pi|_{\Ttilde^r(X_{\xp})}
  \end{equation}
\end{proposition}
\begin{proof}
  Crystalline cohomology base changes like the de Rham cohomology. However, the natural action of $\Frob_{\k}$ on the extended scalars will not be linear. Instead, if $\k'/\k$ is a base extension then the natural action of $\Frob_{\k'}$ will correspond to the linear extension of $\Frob_{\k}^{[\k':\k]}$ from $\H_\crys^{2r}(X_\xp/\Zxp)$. It follows that if $[\k':\k]=u$ then the span of $\Ttilde^r(X_\xp)$ will give the new space of Tate classes. Further extensions of the base field will not increase the dimension of the space of Tate classes. Observing that the obstruction map $\pi_{\cx}$ extends linearly with base change, we conclude the proof using Theorem~\ref{thm:simplified bloch-esnault-kerz-14}.
\end{proof}

In practice we will approximate $\Frob_{\k}$ to a few $p$-adic digits (often $5$ to $20$ digits). Increasing the precision is costly and each power of $\Frob_{\k}$ will lose some of that precision. Considering that $u$ maybe quite large, we will compute $\Ttilde^r(X_\xp)$ in the following way which requires taking at most the $v$-th power of $\Frob_\k$, with $v \colonequals \max_{i} \{\deg \Phi_{i} \mid \gamma_i \neq 0\}$.

\begin{proposition}\label{prop:cylic_splitting}
  Let $T_{i}^r(X_\xp) \colonequals \ker\bigl( \Phi_{i}(q^{-r} \Frob_{\k}) \bigr)$. Then $\Ttilde^r(X_\xp) = \bigoplus_{i\ge 0} T_{i}^r(X_\xp)$, where the sum is taken over $i$ with $\gamma_i \neq 0$.
\end{proposition}
\begin{proof}
  The restriction of $q^{-r} \Frob_{\k}$ to $\Ttilde^r(X_\xp)$ is annihilated by the polynomial $t^u-1$. Therefore, its minimal polynomial will have only reduced cyclotomic factors. Now apply the primary decomposition theorem from linear algebra~\cite[\S 6.8, Theorem 12]{hoffman-kunze--linear_algebra}.
\end{proof}

\subsection{Improving the obstruction using the Frobenius decomposition}\label{sec:improve_obstruction}

The space of ``eventually Tate classes'' $\Ttilde \colonequals \Ttilde^r(X_\xp)$ from \eqref{eq:Ttilde} admits the Frobenius decomposition $\Ttilde = \bigoplus_{i} T_i$ of Proposition~\ref{prop:cylic_splitting}.
It is tempting to obstruct each $T_i$ individually in order to improve the upper bounds we can get for the dimension of $A^r(X_{\overline{K}})$. Assuming the full Tate conjecture, as described below, we show that this is possible here.

The full Tate conjecture assumes, in addition to the Tate conjecture, that the intersection product on algebraic cohomology classes is non-degenerate~\cite{milne-07}.
This additional requirement always holds for surfaces.
In particular, K3 surfaces over finite fields continue to satisfy the full Tate conjecture~\cite{charles-13, madapusi-13, kim-madapusi-16}.

For this section let $A \colonequals A^r(X_{\overline{K}})_\Qq$, $B \colonequals A^r(\overline{X}_\xp)_\Qq$ and $H \colonequals \H^{2r}_{\crys}(X_\xp/\Zxp)\otimes \Qp$.
We warn the reader that $B$ does not map into $H$. Pick a finite extension $\k'/\k$ where all algebraic cycle classes of $\overline{X}_{\xp}$ can be defined. Recall that $W(\k')$ is the ring of Witt vectors of $\k'$ and $Q(\k')$ is its fraction field.
Then $q^{-r} \Frob_\k$ maps $B$ equivariantly into
\begin{equation}\label{eq:hprime}
  H' \colonequals \H^{2r}_{\crys}(X_{\xp} \times_{\k} \k'/ W(\k')) \otimes_{W(\k')} Q(\k') \simeq H \otimes_{\Qxp}  Q(\k').
\end{equation}
The natural action of $\Frob_\k$ on $H'$ is \emph{not} the one that linearly extends $q^{-r}\Frob_\k$ from $H$.
Overcoming this non-linearity is the main technical obstacle in this section. Let $\sigma \colon Q(\k') \to Q(\k')$ be the field homomorphism induced by the Frobenius map of $\k$ on $\k'$. The natural $\Frob_\k$ action on $H'$ is the $\sigma$-linear extension of the $\Frob_\k$ action on $H$.

\begin{lemma}\label{lem:equivariant_inclusion}
  The image of the inclusion $\iota\colon A^r(X_{\overline{K}}) \toi A^r(\overline{X}_p)$ is invariant under the action of $q^{-r} \Frob_{\k}$.
\end{lemma}
\begin{proof}
  We need to find an action on the left hand side that commutes with the Frobenius action on the right hand side. Note that there is a finite Galois extension $K'/K$ such that $A^r(X_{\overline{K}}) = A^r(X_{K'})$. We can also choose a prime $\xq$ of $K'$ lying above $\xp$ for the reductions. It is standard that one can lift the Frobenius action into the subgroup of $\gal(K'/K)$ that fixes~$\xq$~\cite[\S~8]{milneANT}.
\end{proof}

Let $\chi_A, \chi_B \in \Qq[t]$ be the characteristic polynomials of the $q^{-r} \Frob_\k$ action on $A$ and $B$ respectively. We will write $\chi_{\Ttilde} \in \Qq[t]$ for the characteristic polynomial of $q^{-r} \Frob_\k$ on $\Ttilde \subset H$ from~\eqref{eq:Ttilde}. We know that $q^{-r} \Frob_\k$ is of finite order on each of these spaces, therefore we can factor these characteristic polynomials as follows:
\begin{equation}
\begin{aligned}\label{eq:exponents}
  \chi_A =& \prod_{i} \Phi_i^{\alpha_i},\, \chi_B =& \prod_{i} \Phi_i^{\beta_i}, \text{ and } \chi_{\Ttilde} =& \prod_{i} \Phi_i^{\gamma_i},
\end{aligned}
\end{equation}
where $\Phi_i \in \Zz[t]$ is the cyclotomic polynomial for any primitive $i$-th root of unity.
Note that we want to determine the dimension $\dim A$ which equals $\sum_{i} \alpha_i \deg \Phi_i$.

Decompose these spaces into $A = \bigoplus_i A_i$, $B=\bigoplus_i B_i$ and $\Ttilde = \bigoplus_i T_i$ as in Proposition~\ref{prop:cylic_splitting}, with $A_i = \ker\left( \Phi_i(q^{-r} \Frob_\k) \right) $ and so on.

\begin{lemma}\label{lem:exponents}
  For each $i$ we have $\alpha_i \le \beta_i \le \gamma_i$.
\end{lemma}
\begin{proof}
  The first inequality $\alpha_i \le \beta_i$ follows from Lemma~\ref{lem:equivariant_inclusion} since $B_i \subset A_i$. The inequality $\beta_i \le \gamma_i$ is more subtle, but standard. Either one can argue as in Proposition~\ref{prop:decompositions_agree} below, or deduce it from the Weil conjectures in comparison with the \'etale cohomology with $\Qq_{\ell}$-coefficients.
\end{proof}

For each $i$, let $T_i' \simeq T_i \otimes_{\Qxp} Q(\k')$ be the span of $T_i$ in $H'$, see~\eqref{eq:hprime}.

\begin{proposition}\label{prop:decompositions_agree}
Assuming the full Tate conjecture for $X$, we have $T'_i \cap B = B_i$.
\end{proposition}
\begin{proof}
  For each $i$, let $\k_i/\k$ be a degree $i$ extension and let $H_i$ be the crystalline cohomology of $X_{\xp^i} \colonequals X_\xp \times_\k \k_i$. The operator $q^{-ri} \Frob_\k^i - 1$ is linear on $H_i$ and it annihilates $A^r(X_{\xp^i})$. Moreover, $B_{\le i} = \bigoplus_{j \le i} B_j$ coincides with $A^r(X_{\xp^i})$ because both sides are the fixed elements of $A^r(\overline{X}_\xp)$ with respect to the absolute Galois group of $\k_i$.

  Having assumed the full Tate conjecture, we may conclude that $B_{\le i}$ is a $\Qq$-substructure for $T_{\le i}'$. In particular, elements of $B$ are $Q(\k')$-linearly independent in the ambient space $H'$. This implies the equality $\Ttilde'_{\le i} \cap B = B_{\le i}$. The proposition for $i=1$ follows immediately. Fix $i$ and assume that the proposition holds for each $j<i$.

  Although the $\Frob_\k$ action is not $Q(\k')$-linear on $T_i'$, the space $T_i'$ is nevertheless a $\Qq$-vector space that is $\Frob_\k$ invariant. Thus, the intersection $T_i' \cap B$ is a $\Frob_\k$ invariant $\Qq$-subspace of $B_{\le i}$. In particular, it admits a decomposition into a direct sum, with each component lying in one $B_j$ for $j \le i$. But distinct spaces $T_j'$ are disjoint and each $T_j'$ contains $B_j$ when $j<i$. It follows immediately that the intersection $T_i' \cap B$ must lie in $B_i$. The equality follows.
\end{proof}

For each $i \in \Zz_{>0}$ define the integer $r_i$ as
\begin{equation}
  r_i = \left\{
    \begin{array}{lr}
      0 & \pi_{T_i} = 0 \\
      \deg \Phi_i & \pi_{T_i} \ne 0
    \end{array}
    \right. .
\end{equation}

\begin{corollary}\label{cor:better_bound}
  Assuming the full Tate conjecture, $\dim_{\Qq} A^r(X_{\overline{K}}) \le \dim_{\Qxp} \Ttilde - \sum_i r_i$.
\end{corollary}
\begin{proof}
With notation as in~\eqref{eq:exponents} we recall $\dim A = \sum_i \alpha_i \deg \Phi_i$, with corresponding sums for the dimensions of $B$ and $\Ttilde$. Thus, we need only show $\alpha_i < \gamma_i$ when the obstruction map on $T_i$ is non-zero. If $\beta_i = \gamma_i$ then, in light of Proposition~\ref{prop:decompositions_agree}, the span of $B_i$ and $T_i$ must agree in $H'$. If $\pi_{T_i} \neq 0$ then the containment $A_i \subset B_i$ must be strict by Theorem~\ref{thm:simplified bloch-esnault-kerz-14}, implying $\alpha_i < \beta_i = \gamma_i$. On the other hand, if $\beta_i < \gamma_i$ then there is nothing more to do.
\end{proof}

\begin{remark}
  We can use the bound in Corollary~\ref{cor:better_bound} even when computing the obstructions with finite precision. Let $r_{i,N}$ to be $0$ or $\deg \Phi_i$ depending on whether $\pi_{T_i} \equiv 0 \bmod p^N$. Clearly, $r_{i,N} \le r_i$ and $\dim_{\Qxp} \Ttilde - \sum_i r_{i,N}$ also serves as an upper bound.
\end{remark}

The result of Corollary~\ref{cor:better_bound} is easy to use. However, we can improve these bounds by computing more. Define the following map on $T_i$
\begin{equation}\label{eqn:pii}
  \pi_{i} : v \mapsto \left(\pi(v), \pi\circ q^{-r} \Frob_\k(v), \dots, \pi \circ \left(q^{-r} \Frob_\k\right)^{u_i -1}(v)\right).
\end{equation}
Observe that this map is $\Qxp$-linear and let $L_i = \ker \pi_i$. The space $L_i$ is the largest $\Frob_\k$-invariant subspace of $\ker \pi_{T_i}$.

\begin{proposition}\label{prop:even_better_bound}
  Assuming the full Tate conjecture, we have the following inequality
  \[
    \dim_{\Qq} A^r(X_{\overline{K}}) \le \sum_i \dim_{\Qxp} L_i.
  \]
\end{proposition}
\begin{proof}
  Let $u_i$ be the degree of $\Phi_i$. We extend the map $\pi_i$ to the span $T_i'$ of $T_i$ in $H'$ as follows. On each element we use the same definition
  \begin{equation}
    \pi_{i}^! : v \mapsto \left(\pi(v), \pi\circ q^{-r} \Frob_\k(v), \dots, \pi \circ \left(q^{-r} \Frob_\k\right)^{u_i -1}(v)\right).
  \end{equation}
  However, we modify the $Q(\k')$-action on the codomain so that $Q(\k')$ acts $\sigma^i$-linearly on the $i$-th coordinate. As a consequence, $\pi_i^{!}$ is $Q(\k')$-linear. Therefore, the kernel of $\pi_i^{!}$ is the span $L_i'$ of $L_i$ in $H'$.

  Since $A_i \subset B_i$ is in the kernel of $\pi$ and $A_i$ is invariant under Frobenius, $A_i$ is contained in the kernel of $\pi_i^!$. Thus, the $Q(\k')$-span of $A_i$ lies in $L_i'$. Since elements of $B_i$ are $Q(\k')$-linearly independent, the same holds for $A_i$. We conclude that
  $\dim_{\Qq} A_i \le \dim_{Q(\k')} L_i' = \dim_{\Qxp} L_i$.
\end{proof}

\subsection{Bounds on the characteristic polynomial of the Frobenius}
\label{sec:bounds of charpoly}

Let $\overline{X}_\xp$ denote the scheme $X_\xp \times_\k \overline{\k}$ for an algebraic closure $\overline{\k}$ of $\k$. Let $q=p^n$ be the number of elements in $\k$.  The Weil conjectures (now a theorem~\cite{deligne-weil}) tell us that the Hasse--Weil zeta function of $X_\xp$ has the form
\begin{equation}
  Z(X_{\xp}, t) \colonequals \exp\left( \sum_{m=1} ^{\infty} \frac{ \# X_\xp(\Ff_{q^m})}{m} t^m \right) = \prod_{i = 0}^{2 \dim(X_\xp)} P_i{(X_\xp, t)}^{{(-1)}^i},
\end{equation}
where $P_i(X_\xp,t)$ is the reciprocated characteristic  polynomial of the Frobenius action on $\H^i_{\crys}(X_\xp/\Zxp)$.
\begin{equation}
P_i(X, t) \colonequals \det(1 - t \Frob_\k) \in 1 + t\Zz[t].
\end{equation}
The polynomial $P_i(X_\xp,t)$ has integer coefficients, has constant term $1$, and all of its roots in $\Cc$ have Euclidean norm $q^{-i/2}$.  This information is crucial in deducing $P_i$'s from an approximate representation of the Frobenius~(see Section~\ref{sec:compute chi frob}).

Note that the $\chi_{2r}$, defined in Section~\ref{sec:enlarging base field}, and $P_{r}$ are related by
\begin{equation}
  t^{\deg \chi_{2r}}\chi_{2r}(1/t) = P_{2r}(X,t q^{-r}).
\end{equation}

\begin{remark}
  Tate conjecture states that the real roots of $P_{2r}(X_{\xp} \times_{\k} \k',q^{-r}t) \in \Qq[t]$ for a sufficiently large (but finite) extension $\k'/\k$ coincides with $\dim A^r(\overline{X}_{\xp})$.
  Thus, we expect the parity of $\dim A^r(\overline{X}_{\xp})$ to match the parity of $\dim \H_{\crys}^{2r}(X_{\xp}/\Zxp) = \deg P_{2r}$.
\end{remark}

\section{Computing in crystalline cohomology}\label{sec:computing_in_crystalline}

Computations in crystalline cohomology are made possible by comparing it with two other cohomology theories. Berthelot's comparison theorem relates crystalline cohomology to a characteristic~$0$ de~Rham cohomology. Monsky--Washnitzer cohomology, on the other hand, allows one to represent the Frobenius map explicitly. In this section, we outline this construction.

The de~Rham cohomology is amenable to computation in general.
For hypersurfaces in particular, Griffiths' explicit description of a basis makes these computations highly practical, see Section~\ref{sec:griffiths}.  The approximations of the action of the Frobenius are made in terms of the this basis.

This approach to approximating the Frobenius action was first conceptualized by Kedlaya~\cite{kedlaya--practice}, first implemented by Abbott, Kedlaya and Roe~\cite{abbott-kedlaya-roe-10}, and shown to be practical in larger characteristic by Costa and Harvey~\cite{harvey-hypersurface1, harvey-hypersurface2, harvey-hypersurface3, costa-phd}.

We summarize the required statements from~\cite{abbott-kedlaya-roe-10} and re-frame some using $\Zxp$ coefficients as opposed to $\Qxp$ coefficients (e.g. Proposition~\ref{prop:gamma_equals_Hlog}).

\subsection{Crystalline to de~Rham cohomology}\label{sec:crystalline_to_dR}

The de~Rham cohomology of an affine variety does not behave well in finite characteristic. On the other hand, in characteristic zero, the cohomology of a hypersurface is made explicit by working with the (affine) complement of the hypersurface. The de~Rham cohomology of a \emph{log pair} is what is needed to carry the characteristic zero advantages to positive characteristic.

Let $\cy$ be a smooth proper variety over $\Zxp$ and let $\cx \toi \cy$ be a subvariety that is a relative normal crossing divisor. Such a pair $(\cy,\cx)$ is called a \emph{smooth proper pair} over $\Zxp$~(\cite[Section~2.2]{abbott-kedlaya-roe-10}). Denote the complement of $\cx$ in $\cy$ by $\cu \colonequals \cy \setminus \cx$. In practice, we will take $\cy = \Pp_{\Zxp}^{2r+1}$ and $\cx$ a smooth hypersurface.

The de~Rham cohomology, as well as the crystalline cohomology, of a smooth proper pair $(\cy,\cx)$ is well defined~\cite[Chapter 2]{shiho-02}.
The cohomology of the hypersurface complement $\cu$, at least over the generic fiber, can be computed via the pair $(\cy,\cx)$.
Indeed, there exists a natural isomorphism~\cite[Theorem 6.4]{kato-89}:
\begin{equation}
  \H_{\dR}^i((\cy,\cx)/\Zxp)\otimes \Qxp \simeq \H_{\dR}^i(\cu_{\Qxp}/\Qxp).
\end{equation}

We may compare the de~Rham and crystalline cohomologies of a smooth proper pair $(\cy,\cx)$ by the following generalization of Berthelot's comparison theorem. We denote by $\cy_\xp$ and $\cx_\xp$ the special fibers of $\cy$ and $\cx$ over $\Zxp$.

\begin{theorem}[Berthelot comparison theorem~\cite{berthelot-97, shiho-02}]
  \label{thm:comparison integral}
  For each $i\ge 0$, there is a canonical isomorphism
  \begin{equation*}
    H^{i}_\dR ((\cy, \cx)/\Zxp) \simeq H^{i}_\crys (\cy_\xp, \cx_\xp).
  \end{equation*}
The functoriality of crystalline cohomology thus equips the de~Rham cohomology of $(\cy,\cx)$ with a Frobenius action.
\end{theorem}

Let us remark that if $\cx=\emptyset$ then $\H^i_{\dR}((\cy,\cx)/\Zxp)$ is denoted by $\H^i_{\dR}(\cy/\Zxp)$ as it coincides with the usual de~Rham cohomology of $\cy$. The analogous notational convention holds for the crystalline cohomology.

\subsection{Splitting the cohomology of a hypersurface}

We will first recall the characteristic zero statement regarding the splitting of the cohomology of a smooth hypersurface. Then, we will give the equivalent statement over $\Zxp$.

\subsubsection{Splitting in characteristic zero} Let $K$ be a characteristic zero field and $X \subset \Pp^{n+1}_K$ a smooth hypersurface with complement $U$. By the Lefschetz hyperplane theorem and Poincar\'e duality, the only non-trivial cohomology group of $X$ is in degree $n$. Moreover, there is a natural splitting of the cohomology of $X$ which is orthogonal with respect to the cup product:
\begin{equation}
  \H^n_{\dR}(X/K) \simeq \H^{n+2}_{\dR}(\Pp_K^{n+1}/K) \oplus^{\perp} \H^{n+1}_{\dR}(U/K).
\end{equation}
The projection onto the first factor, the cohomology of $\Pp_K^{n+1}$, has the following geometric interpretation. Thinking of the underlying complex analytic manifolds, there is a natural map from the singular homology of $\Pp^{n+1}_{\Cc}$ to the homology of $X_{\Cc}$ obtained by intersecting a homology class in $\Pp^{n+1}$ with $X$. The dual of this map gives the projection onto the first factor above. The kernel of this projection is called the \emph{primitive part of the cohomology} and is canonically identified with $\H^{n+1}_{\dR}(U)$. Note that if $n$ is odd then $\H^{n+2}_{\dR}(\Pp_K^{n+1})$ is zero.

\subsubsection{Splitting over the ring of Witt vectors} We will now work over the Witt vectors $\Zxp$ of the residue field $\k$ of characteristic $p$. We take $\cx \toi \Pp^{2r+1}_{\Zxp}$ to be a smooth hypersurface of even dimension, the odd case being much simpler and of less interest for our purposes. Recall that the de~Rham cohomology groups of  $\Pp_{\Zxp}^{2r+1}$, $\cx$, and $(\Pp_{\Zxp}^{n+1},\cx)$ admit canonical Frobenius actions via Theorem~\ref{thm:comparison integral}.  The goal of this section is to state and prove Proposition~\ref{prop:split_cohomology}, which is not explicitly stated in~\cite{abbott-kedlaya-roe-10} but follows from the arguments presented there.

To put Proposition~\ref{prop:split_cohomology} in context, let us recall some of the simpler results. From~\cite[Corollary~3.1.4]{abbott-kedlaya-roe-10} we have
\begin{equation}\label{eq:cohom_of_Pn}
  \H_{\dR}^{i}(\Pp^{n}_{\Zxp}) = \left\{
    \begin{array}{lr}
      \Zxp & i=\{0,2,\dots,2n\} \\
      0   & \text{otherwise}
    \end{array},
    \right. \qquad \forall n \ge 0.
\end{equation}
Furthermore, by arguing as in~\cite[Lemma~2.4]{pan-20} or \cite{deligne-illusie}, we see that the cohomology groups $\H^i_{\dR}(\cx/\Zxp)$ are $\Zxp$-torsion-free for all $i$. Comparing with characteristic zero, we conclude that:
\begin{equation}
  \H_{\dR}^{i}(\cx) = \left\{
    \begin{array}{lr}
      \Zxp & i\in \{0,2,\dots,2n\}\setminus \{2r\} \\
      0   & \text{otherwise}
    \end{array}
    \right. .
\end{equation}

\begin{proposition}\label{prop:split_cohomology}
  Suppose the degree of $\cx$ is not divisible by $p$. Then, there is a natural Frobenius equivariant splitting of the cohomology of $\cx$:
  \begin{equation*}
    \H_{\dR}^{2r}(\cx/\Zxp) \simeq \H_{\dR}^{2r+2}(\Pp^{2r+1}_{\Zxp}/\Zxp)(1) \oplus \H_{\dR}^{2r+1}((\Pp_{\Zxp}^{n+1},\cx)/\Zxp)(1).
  \end{equation*}
\end{proposition}
\begin{proof}
  Let $\cy=\Pp_{\Zxp}^{2r+1}$.
  There is a long exact sequence associated to proper pairs~\cite[Proposition~2.2.8]{abbott-kedlaya-roe-10}. Combining with~\eqref{eq:cohom_of_Pn} and ignoring the Frobenius action we get the following exact sequence:
  \begin{equation}\label{eq:mini_long_exact_sequence}
    0 \too \H_{\dR}^{2r+1}((\cy,\cx)/\Zxp) \too \H_{\dR}^{2r}(\cx/\Zxp) \overset{\varsigma}{\too} \H_{\dR}^{2r+2}(\cy/\Zxp).
  \end{equation}
  Corresponding to the inclusion $\iota\colon \cx \toi \cy$ we have pullback maps on cohomology
  \begin{equation}
  \iota^*_j \colon \H_{\dR}^{j}(\cy/\Zxp) \to \H_{\dR}^{j}(\cx/\Zxp).
  \end{equation}
  Let $h$ be the first Chern class of the line bundle $\co_{\cy}(1)$ and $h^k$ its $k$-th self product.
  It is easy to see that $\varsigma\circ \iota^*_{2r}(h^r) = \deg(\cx) h^{r+1}$. Since powers of $h$ generate the cohomology of $\cy$, and because $p$ does not divide the degree of $\cx$, the composition $\varsigma \circ \iota^*_{2r}$ is an isomorphism. In particular, the last arrow in the sequence~\eqref{eq:mini_long_exact_sequence} is a surjection and the sequence splits.

  To make the splitting equivariant with respect to the action of the Frobenius we must twist the two components. Recall that the cohomology of $\cx$ is torsion-free. Thus each component injects into its tensor with $\Qxp$. Now use the discussion following Definition~2.3.3 and Proposition~2.4.1 itself in ~\cite{abbott-kedlaya-roe-10} to conclude the proof.
\end{proof}

\subsection{Torsion-free obstruction space}\label{sec:torsion_free_obstruction}

Let $\F^\bullet\H_{\dR}^{2r}(\cx/\Zxp)$ denote the (decreasing) Hodge filtration on the cohomology of the smooth hypersurface $\cx \subset \Pp^{2r+1}$.

\begin{proposition}\label{prop:torsion_free_quotients}
  For any $a,b\ge0$ the quotient
  \begin{equation}\label{eq:quotients}
    \F^{a}\H_{\dR}^{2r}(\cx/\Zxp)/ \F^{a+b}\H_{\dR}^{2r}(\cx/\Zxp)
  \end{equation}
  is $\Zxp$-torsion-free.
\end{proposition}
\begin{proof}
  The Hodge to de~Rham spectral sequence degenerates at $E_1$~\cite[Theorem~4.2.2]{deligne-illusie, illusie-91}.
  Therefore, for any $s\ge 0$, the graded piece
  \begin{equation}
    \Gr_s \colonequals \F^{s}\H_{\dR}^{2r}(\cx/\Zxp)/ \F^{s+1}\H_{\dR}^{2r}(\cx/\Zxp)
  \end{equation}
  is isomorphic to $\H^{2r-s}(\cx,\Omega_{\cx/\Zxp}^s)$. Because the Hodge numbers of a smooth hypersurface depends only on degree and dimension (and not on the base field), we conclude that the graded pieces, $\Gr_s$, are torsion-free.

  The Hodge filtration on cohomology descends to a filtration on~\eqref{eq:quotients} whose non-zero graded pieces coincide with the graded pieces, $\Gr_{\bullet}$, of the cohomology. A filtered module with torsion-free graded pieces is torsion-free.
\end{proof}

Recall the set-up from Section~\ref{sec:lifting_cycles}.
Let $\pi_{T}$ be the obstruction map~\eqref{eq:intro_pi} from $T^r(X_\xp)$ and let $\pi_A$ be the obstruction map from $A^r(X_\xp)$. Proposition~\ref{prop:torsion_free_quotients} implies the following result.

\begin{corollary}\label{cor:torsion_free_obstruction}
  A Tate class $v \in T^r(X_\xp)$ is unobstructed if and only if a non-zero integral multiple of $v$ is unobstructed. In symbols:
  \[
    \pi_T(v) = 0 \iff \left( \exists n\in \Zz_{>0},\, \pi_T(nv)=0 \right).
  \]
 The equivalent formulation using $\pi_A$ also holds. \qed
\end{corollary}

By itself, Corollary~\ref{cor:torsion_free_obstruction} does not imply that the cokernel of the inclusion map $A^r(\widehat \cx) \toi A^r(X_{\xp})$ is $\Zz$-torsion-free. This is because the theorem of Bloch--Esnault--Kerz (Theorem~\ref{thm:simplified bloch-esnault-kerz-14}) works with rational coefficients and thus implies only that a \emph{multiple} of an obstructed class in $A^r(X_\xp)$ can formally lift.

Nevertheless, for divisors on surfaces ($r=1$), it is known that the cokernel of $A^1(X) \to A^1(X_{\xp})$ is torsion-free. This follows from the much stronger~\cite[Theorem~4.1.2.1]{raynaud--picard} or its generalization~\cite[Theorem~1.4]{elsenhans-jahnel-11b}.

\subsection{Primitive cohomology after Griffiths}\label{sec:griffiths}
In this section we will describe Griffiths' basis for the primitive cohomology $\Hprim \subset \H_{\dR}^n(\cx/\Zxp)$ of a smooth hypersurface $\cx \subset \Pp^{n+1}_{\Zxp}$ (Definition~\ref{def:griffiths_basis}). The main result leading up to this definition is Proposition~\ref{prop:gamma_equals_Hlog}.

Let $\Pp$ denote $\Pp^{n+1}_{\Zxp}$ throughout this section and let $d$ be the degree of $\cx$. Recall that $\Hprim$ is naturally isomorphic to $\Hlog \colonequals \H_{\dR}^{n+1}((\Pp,\cx)/\Zxp)$. When tensored with $\Qxp$, this logarithmic cohomology coincides with the cohomology of the generic fiber $U$ of the complement $\cu \colonequals \Pp \setminus \cx$ of our hypersurface $\cx$. That is,
\begin{equation}\label{eq:log_vs_U}
  \Hlog \otimes_{\Zxp} \Qxp \simeq \H_{\dR}^{n+1}(U/\Qxp).
\end{equation}
The $(n+1)$-th de~Rham cohomology of the affine variety $U$ is readily computed. Consider the following map, induced by exterior differentiation on global sections,
\begin{equation}
  F \colonequals \H^0(\Pp, \Omega^{n}_{\Pp/\Zxp}(n\cdot \cx))  \overset{\dd}{\too} G \colonequals \H^0(\Pp, \Omega^{n+1}_{\Pp/\Zxp}((n+1)\cdot \cx)) .
\end{equation}
We will denote the quotient $G/\dd F$ by $\Gamma$. Using characteristic 0 arguments~\cite[\S 6.1]{voisin-2007-volII}, we know $\Gamma \otimes_{\Zxp} \Qxp \simeq \H_{\dR}^n(U/\Qxp)$. In light of~\eqref{eq:log_vs_U}, we will compare $\Hlog$ with $\Gamma$.

We have a filtration $G^\bullet$ on $G$ induced by the pole order of forms:
\begin{equation}
  G^j \colonequals \im \left(  \H^0(\Pp, \Omega^{n+1}_{\Pp/\Zxp}((n+1-j)\cdot \cx)) \to \H^0(\Pp, \Omega^{n+1}_{\Pp/\Zxp}((n+1)\cdot \cx))\right) .
\end{equation}
The induced filtration on $\Gamma$ will be denoted by $\Gamma^\bullet$, so that $\Gamma^j = \im \left( G^j \to \Gamma \right) $.

Let $R = \Zxp[x_0,\dots,x_{n+1}]$ and $f \in R$ be a polynomial defining $\cx$.
Let $J = (\del_0 f,\dots,\del_{n+1} f)$ in $R$ be the Jacobian ideal of $f$, where $\del_i$ denotes differentiation by $x_i$. We will write $S$ for the quotient $R/J$. Homogeneous components of $R,S,J$ will be denoted by subscripts. Let $N_j \colonequals (n+1-j)d - n-2$ for $j=0,\dots,n$. We have a natural isomorphism
\begin{equation}\label{eq:R_to_G}
  R_{N_j} \isoto G^j : p \mapsto \frac{p}{f^{n+1-j}}\vol_{\Pp},
\end{equation}
here $\vol_{\Pp}$ is a natural generator of $\Omega^{n+1}_{\Pp/\Zxp}(n+2)$ given by
\begin{equation}
  \vol_{\Pp} \colonequals \sum_{i=0}^{n+1} (-1)^{i} x_i \dd x_0 \wedge \dots \wedge \widehat{\dd x_i} \wedge \dots \wedge \dd x_{n+1}.
\end{equation}

\begin{lemma}\label{lem:graded_gamma}
  If $p > n+1$ and $p \nmid d$ then the map $R_{N_j} \tos \Gamma^j$ induces an isomorphism
  \[
    \left( R/J \right)_{N_j} \isoto \Gamma^j/\Gamma^{j+1}.
  \]
\end{lemma}
\begin{proof}
  This is standard in characteristic 0~\cite{voisin-2007-volII,griffiths--periods}. The image of $F \to G$ is generated by elements of the form
\begin{equation}
  \left( \frac{g}{f^m} - \frac{g f}{f^{m+1}}  \right) \vol_{\Pp}
  \quad \text{and} \quad
  \left( \frac{ \partial g / \partial x_i }{f^m} -  m\frac{g \partial f / \partial x_i}{f^{m+1}} \right)  \vol_{\Pp}
\end{equation}
each $m=1,\dots,n+1$ and $g\in R_{N_{n+1-m}}$. Since $p \nmid d$, $f \in J$ and since $p > n+1$ we can divide by $m$ in the relations above. It follows that, arguing as in~\cite[\S 4]{griffiths--periods}, the pole order of an element in $G$ can be reduced modulo $\dd F$ if and only if it is in the image of $J$.
\end{proof}

\begin{lemma}\label{lem:S_tors_free}
  If $p \nmid d$ then $R/J$, and in particular $\left( R/J \right)_N$ for each $N \ge 0$, is $\Zxp$-torsion-free.
\end{lemma}
\begin{proof}
  Let $K$ stand for either the residue field or of the function field of $\Zxp$. In either case, with $S=R/J$, $S \otimes_{\Zxp} K$ is a local complete intersection ring of dimension zero. Since $J\otimes K$ is generated by $n+2$ elements forming a regular sequence each of degree $d-1$, the Hilbert series of $S\otimes K$ is independent of $K$. We used the smoothness of $\cx$ here and that $f \in J$ when $p \nmid d$. Since $\dim_K S\otimes K$ is independent of $K$, $S$ is torsion-free.
\end{proof}

\begin{lemma}
  If $p \nmid d$ and $p > n+1$ then $\Gamma$ is torsion-free.
\end{lemma}
\begin{proof}
  Lemmas~\ref{lem:graded_gamma} and~\ref{lem:S_tors_free} imply that the graded pieces of $\Gamma$ are torsion-free.
\end{proof}

In light of the three lemmas above, we conclude that $\Gamma$ injects into $\Gamma \otimes_{\Zxp} \Qxp$, which in turn is isomorphic to $\Hlog \otimes \Qxp$ and $\Hprim \otimes \Qxp$.  Recall that the de~Rham cohomology of $\cx$ is torsion-free and by Proposition~\ref{prop:split_cohomology} so is $\Hlog$. Thus we have two injections:
\begin{equation}\label{eq:two_injections}
  \Gamma \toi \Hprim \otimes \Qxp \hookleftarrow \Hlog, \quad \text{if } p\nmid d ,\, p> n+1.
\end{equation}

\begin{proposition} \label{prop:gamma_equals_Hlog}
  If $p \nmid d$ and $p> n+1$ then the images of $\Gamma$ and $\Hlog$ in~\eqref{eq:two_injections} coincide. In particular, $\Gamma$ is isomorphic to $\Hlog$.
\end{proposition}
\begin{proof}
  This is now a consequence of Remark~3.4.4 and Corollary~3.4.7 of~\cite{abbott-kedlaya-roe-10}.
\end{proof}

Impose the grevlex monomial ordering on $R$ and on $R \otimes \k \simeq \Ff_q[x_0,\dots,x_{n+1}]$. As a result, there is a natural ordered monomial $\k$-basis $B''$ for the quotient $R/J \otimes \k$. Let $B' \subset R$ be monomials that map to $B''$. By Nakayama's lemma, $B'$ is a $\Zxp$-basis for the free-module $R/J$.
We will write $B \subset \Hprim$ for the image of $B' \cap \bigoplus_{j=0}^{n} \left( R/J \right)_{N_j}$ with respect to the following composition of maps:
\begin{equation}
 \bigoplus_{j=0}^{n} \left( R/J \right)_{N_j} \tos \Gamma \isoto \Hprim.
\end{equation}

\begin{definition}\label{def:griffiths_basis}
  When $p \nmid d$ and $p > n+1$ then $B \subset \Hprim$ will be called \emph{the Griffiths basis} for the primitive cohomology of $\cx$.
\end{definition}

A particularly potent feature of the Griffiths basis is that it ``respects the Hodge filtration'' in the following sense. Write $B = (w_1,\dots,w_s)$ in increasing grevlex ordering. If $\F^\bullet \Hprim$ is the Hodge filtration on the primitive cohomology then
\begin{equation}
  \forall j =0,\dots,n,\, \exists i_j,\,\, \text{s.t.}\,\, \F^j \Hprim = \Zxp \langle w_1,\dots,w_{i_j} \rangle.
\end{equation}
This follows from the fact that pole order filtration on $\Gamma\otimes \Qxp$ coincides with the Hodge filtrations on $\H_{\dR}^n(U/\Qxp)$ and on $\Hprim \otimes \Qxp$~\cite{voisin-2007-volII}.

\subsubsection{Approximating the Frobenius matrix in terms of the Griffiths basis}\label{sec:approximate_frobenius}

We comment briefly on how to compute an approximation of Frobenius action $\Frob_\k$ on the primitive cohomology $\Hprim$ of a hypersurface $\cx$. We will focus on the even dimensional case, but the odd dimensional case requires minimal change.
We recommend~\cite{abbott-kedlaya-roe-10}, \cite{costa-phd} or~\cite{chk} for a more detailed discussion.
We assume $p \nmid d$ and $p > n+1$ throughout.

Recall $\cu = \Pp \setminus \cx$, $U = \cu \otimes \Qxp$ and that the de~Rham cohomology $\Hu \colonequals \H^{n+1}_{\dR}(U/\Qxp)$ of $U$ is identified with $\Hprim \otimes \Qxp$. Note also that $\Hprim$ is torsion-free and the natural inclusion
\begin{equation}
  \Hprim(-1) \toi \Hprim \otimes \Qxp \isoto \Hu
\end{equation}
is equivariant with respect to the Frobenius action (\cite[Proposition~2.4.1]{abbott-kedlaya-roe-10}).

In the previous section, we represented elements of $\Hprim$ and of $\Hu$ using polynomials. We will work with the Griffiths basis $B \subset \Hprim$. Let $x^{\beta_1},\dots, x^{\beta_s} \in \Zxp[x_0,\dots,x_{n+1}]$ be the monomials that map to $B$. The image of $B$ in $\Hu$ is given by~\eqref{eq:R_to_G}. Namely, with $l_i = \big(\deg\big(x^{\beta_i}\big)+n+2\big)/d$ the following elements give the corresponding basis in $\Hu$:
\begin{equation}
  \eta_i \colonequals \frac{x^{\beta_i}}{f^{\ell_i}}\vol_{\Pp} \bmod G,\quad i=1,\dots,s.
\end{equation}

Here, one switches to another cohomology theory. Since $U$ is affine, $\Hu$ can be computed using the Monsky--Washnitzer cohomology. The advantage of Monsky--Washnitzer cohomology in this context is that it provides great flexibility in how one can choose to represent the Frobenius matrix in the chain of forms that compute the cohomology~\cite[Definition 2.4.2]{abbott-kedlaya-roe-10}.

The Frobenius action can be described on $\eta_i$ and Griffiths--Dwork reduction allows one to re-cast $\Frob_\k(\eta_i)$ in terms of the Griffiths basis again. The catch is that $\Frob_\k(\eta_i)$ becomes an infinite sum, each term having coefficients with higher and higher valuation. Truncating the infinite sequence and applying the Griffiths--Dwork reduction gives an approximation of $\Frob_\k(\eta_i)$ in the Griffiths basis. The degree of truncation required to attain the desired precision is discussed in~\cite[\S 3.4]{abbott-kedlaya-roe-10}.

\section{Main Algorithm}\label{sec:algorithm}

This section states and explains the steps of the main algorithm of this paper, Algorithm~\ref{alg:main}.
The version here provides an upper bound for the geometric middle Picard number of a given smooth hypersurface.
A simple variation gives bounds on the geometric Picard number of a Jacobian, see~Section~\ref{sec:jacobian}.
More sophisticated variations allow for the study of non-degenerate hypersurfaces in simplicial toric varieties~\cite{chk}.

The algorithm takes the equation $f \in K[x_0,\dots,x_{2r+1}]$ of a smooth hypersurface $X \subset \Pp^{2r+1}_K$ as input.
The output is an upper bound for the geometric middle Picard number, $\rho^r(X \times_K \overline{K})$, of~$X$.
Optionally, one may choose to specify a lower bound for the precision $N$ used to approximate the Frobenius map. For small $N$, the upper bound may improve as $N$ is increased. The bound will eventually stabilize. One may also provide a lower bound on the characteristic $p$ of the prime $\xp \subset \co_K$ to be used for the good reduction of $X$.

\stepcounter{ouralgo}

\begin{algorithm}[hb]
  \SetAlgoRefName{\theouralgo}
  \DontPrintSemicolon
  \SetAlgoVlined
  \SetCommentSty{}
  \SetKwInOut{Input}{Input}
  \SetKwInOut{Output}{Output}
  \SetKwFunction{nextgoodprime}{next\_good\_prime}
  \SetKwFunction{griffithsbasis}{griffiths\_basis}
  \SetKwFunction{frobenius}{frobenius}
  \SetKwFunction{lcmrootorder}{lcm\_root\_order}
  \SetKwFunction{cp}{weil\_characteristic\_polynomial}
  \SetKwFunction{BoundRank}{BoundRank}
  \SetKwFunction{tatebasis}{tate\_basis\_matrix}
  \SetKwFunction{projection}{projection\_map}
  \SetKwFunction{Rank}{bound\_rank}
  \SetKwFunction{cycfac}{cyclotomic\_factors}
  \SetKwFunction{stack}{column\_matrix}

  \SetKwFunction{Max}{max}
  \SetKwData{primebound}{char\_bound}
  \SetKwData{precbound}{precision\_bound}
  \SetKwData{dimTi}{dimTi}
  \SetKwData{bound}{bound}
  \SetKwProg{Function}{function}{}{end function}
  \Input{
    $f \in \co_K[x_0, x_1, \dots, x_{2r +1}]$; Optional: $\precbound, \primebound \in \Nn$
  }
  \Output{
  An upper bound for the geometric middle Picard number of $X=Z(f)$.
    }
    \Function{\BoundRank{$f; \precbound=1, \primebound=3$}}{

      \tcp{Pick a good prime, with residue characteristic at least \primebound, \S~\ref{sec:pick a good prime}}
      $\xp \colonequals \nextgoodprime{f, \Max{2r + 6, \primebound}}$\;

      \tcp{Compute a Griffiths basis for primitive cohomology, \S~\ref{sec:compute griffiths}}
      $B \colonequals \griffithsbasis{f, $\xp$}$\;

      \tcp{Compute a minimal working precision, \S~\ref{sec:working prec}}
      $q := \# \k$, $m \colonequals \# B + 1$\;
      $N \colonequals $ \Max$\left(\precbound, \log_p \left( 2 \binom{m}{\left\lceil m/2 \right\rceil}{q^{r \left\lceil m/2 \right\rceil }}\right)\right)$
      \tcp*[l]{See equation~\eqref{eqn:weilbound}}

      \tcp{Compute an $N$-digit approximation of the Frobenius matrix, \S~\ref{sec:compute frob}}
      $[ \Frob_{\k,N} ] \colonequals \frobenius{f, $\xp$, N, B}$\;

      \tcp{Compute the characteristic polynomial of the Frobenius matrix, \S~\ref{sec:compute chi frob}}
      $\chi(q^{-r}\Frob_\k) \colonequals $\cp{$\Frob_{\k,N}, N$}\;

      \tcp{Represent the obstruction map, \S~\ref{sec:compute obstruction map}}
      $[\pi_{\cx}] \colonequals \projection{B}$
      \tcp*[l]{$\pi_{\cx} \colon \H^{2r} _{\dR} (\cx) \to \H^{2r} _{\dR} (\cx) / \F^{r} \H^{2r} _{\dR} (\cx)$}

      \tcp{Loop over the cyclotomic factors of $\chi(q^{-r} \Frob_\k)$, \S~\ref{sec:factor chi frob}}
      $\bound \colonequals 0$ \;
      \For{$\Phi_{i}^{\gamma_i} \ \in \cycfac(\chi(q^{-r} \Frob_\k))$}{

        \tcp{Approximate the space of Tate classes, \S~\ref{sec:approximate tate}}
        $B_{T_{i, N}} \colonequals $\tatebasis{$\Phi_{i}, \Frob_{\k,N}, N$} \tcp*[l]{columns form a basis for~$T_{i, N}$}

        \tcp{Approximate the map $\pi_i$, \S~\ref{sec:approximate pii}}
        $[ \pi_{i, N} ] \colonequals [\pi_{\cx} ] \cdot $ \stack{$1 , [\Frob_{\k,N}] , \cdots, [\Frob_{\k,N}  ]^{\deg \Phi_i -1}$}$ \cdot B_{T_{i, N}} $\;

        \tcp{Increment by the bound on the dimension of $L_i \colonequals \ker \pi_i$,  \S~\ref{sec:dim Li}}
        $\bound \colonequals \bound + \deg \Phi_i ^{\gamma_i} - \Rank([\pi_{i, N}])$ \\
  }
    \Return $\bound$.
  }
    \caption{Computing an upper bound for the geometric middle Picard number.}
    \label{alg:main}
  \end{algorithm}

\subsection{Clarification of the steps in the algorithm}

\subsubsection{Pick a good prime}\label{sec:pick a good prime}
Pick the first prime number $p$ exceeding $2r+6$ and {\tt char\_bound}. Choose an unramified prime $\frakp$ of $K$ lying above $p$. By clearing the denominators of the polynomial $f$ defining $X$, create a model $\cx/\Zxp$ of $X$.  Check if the model $\cx/\Zxp$ is smooth. If not, pick another prime and repeat.

It is possible to eliminate the simple but haphazard searching method above: Compute the discriminant of $X$ over $\Zz$ and avoid the primes dividing this discriminant. However, discriminants tend to be huge.

\subsubsection{Compute a Griffiths basis for primitive cohomology}\label{sec:compute griffiths}
In Section~\ref{sec:griffiths} we describe how to compute a Griffiths basis for the primitive part of $\H^{2r} _\dR (\cx/\Zxp) $.
Choosing the polarization to complete the basis, we represent the arithmetic
Frobenius map $\Frob_\k$ on the cohomology $\H^{2r} _{\dR}(\cx/\Zxp) \simeq \H^{2r} _\crys(X_\frakp/\Zxp)$ via the Berthelot comparison theorem, see Theorem~\ref{thm:comparison integral}.
This square matrix with $\Zxp$ entries will be denoted by $[ \Frob_\k ]$.

\subsubsection{Compute a minimal working precision}\label{sec:working prec}
The dimension of the middle cohomology of a hypersurface of dimension $2r$ and degree $d$ is given by the formula
\begin{equation}
m \colonequals\frac{(d-1)^{2 r+2}+2 d-1}{d}.
\end{equation}
In any case, $m=\# B + 1$ where $B$ is the Griffiths basis for the primitive cohomology.

The Weil conjectures imply that the reciprocal characteristic polynomial of $\Frob_\k$ on $\H^{2r} _{\dR}(\cx/\Zxp)$ has integer coefficients with constant term equal to $1$, i.e.,
\begin{equation}
  P_{2r}(X ,t) \colonequals \det(1 - t \Frob_\k) \in 1 + t \Zz[t].
\end{equation}
Moreover, this polynomial is completely determined by the coefficients of $t^i$ with $i = 0, \dots, \left\lceil m/2 \right\rceil$, and these coefficients have absolute value at most
\begin{equation}
  \label{eqn:weilbound}
\binom{m}{\left\lceil m/2 \right\rceil}{q^{r \left\lceil m/2 \right\rceil }}.
\end{equation}
Thus if $p^N$ exceeds twice this bound, then $P_{2r}(X ,t)$ is determined by its reduction modulo $p^N$.
The bound above might be significantly improved by employing Newton identities, see~\cite[slide 8]{kedlaya-oxfordtalk}.
If the precision requested by the user is not sufficient to lift $P_{2r}(X ,t)$, we increase it accordingly.

\subsubsection{Compute an $N$-digit approximation of the Frobenius matrix}\label{sec:compute frob}
With $N$ as above, we may compute a matrix $[\Frob_{\k,N}]$ approximating the matrix $[\Frob_\k]$ to $N$ $p$-adic digits as explained in~\cite{abbott-kedlaya-roe-10}, \cite{costa-phd} or~\cite{chk}. We sketched the idea in Section~\ref{sec:approximate_frobenius}.
In practice, one may use the library \verb|controlledreduction|\footnote{This library is made available in \texttt{SageMath}~\cite{sage} through the wrapper \url{https://github.com/edgarcosta/pycontrolledreduction}.}.

\subsubsection{Compute the characteristic polynomial of the Frobenius matrix}
\label{sec:compute chi frob}
As explained above we may deduce $P_{2r}$ from $\det(1 - t\Frob_{\k,N})$.
In practice, we compute $\chi(\Frob_{\k,N})(t) \equiv \chi(\Frob_\k)(t) \bmod{p^N}$ and lift each coefficient to the interval $[-p^N /2, p^N /2]$.
By the discussion above, the representatives of the last coefficients of $t^i$ for $i = \left\lceil m/2 \right\rceil, \ldots, m$ match exactly their lifts, and the remaining coefficients are deduced using the functional equation
\begin{equation}
\chi(\Frob_\k)(t) = \pm \bigl(q^{-r} t\bigr)^m \chi(\Frob_\k)\bigl(q^{2r}/t \bigr).
\end{equation}
The sign of the functional equation is the sign of $\det(\Frob_k)$, which we are able to compute to one significant $p$-adic digit.

  \subsubsection{Represent the obstruction map}\label{sec:compute obstruction map}
  The obstruction map $\pi_{\cx}$ (see \S~\ref{sec:obstruction_map}) annihilates the polarization.
  Thus, it remains to describe $\pi_{\cx}$ on the primitive cohomology.
  Since the Griffiths basis on the primitive cohomology respects the filtration (see \S~\ref{sec:griffiths}), the map $\pi_{\cx}$ on the primitive cohomology in the Griffiths basis is just the projection onto the last few coordinates. Let $[\pi_{\cx}]$ be the matrix representation of this projection.

\subsubsection{Extract cyclotomic factors from the characteristic polynomial}\label{sec:factor chi frob}
Factorize the characteristic polynomial over $\Qq[t]$ as in~\eqref{eqn:Qfact}.

    \subsubsection{Approximate the space of Tate classes}\label{sec:approximate tate}
    Let $\Phi_i(t)$ be a cyclotomic factor computed in the previous step, let $B_{T_{i, N}}$ be a basis for the eigenspace $T_{i, N} \colonequals \ker \bigl(\Phi_i \bigl( q^{-r} \Frob_{\k,N} \bigr)\bigr)$.
      This $T_{i, N}$ approximates $T^r _i(X_\xp) \colonequals \ker \bigl(\Phi_i \bigl( q^{-r} \Frob_{\k} \bigr)\bigr)$ from Proposition~\ref{prop:cylic_splitting}.

      The kernels are computed using standard algorithms that keep track of the $p$-adic precision, see~\cite{caruso-roe-vaccon-18}. The computations equip $T_{i,N}$ with a basis which we record in the columns of a matrix $B_{T_{i, N}}$.

\subsubsection{Approximate the map \texorpdfstring{$\pi_i$}{pi_i}}\label{sec:approximate pii}
  Restricting $\pi_{\cx}$ on to $T_{i,N}$, we obtain the approximation
    \begin{equation}
      \pi_{T_{i,N}} \colon T_{i,N} \to \H^{2r} _{\dR} (\cx) / \F^{r} \H^{2r} _{\dR} (\cx)
    \end{equation}
    of the obstruction map $\pi_{T_i}$ on $T^r _i(X_\xp)$.
    Therefore, the matrix representation $[\pi_{T_{i,N}}]$ of $\pi_{T_{i,N}}$ is then equal to $[\pi_{\cx}] \cdot B_{T_{i, N}}$.
    Similarly, we may approximate the map
  \begin{equation}
    \begin{aligned}
      \pi_i \colon T_i ^r (X_p) &\to \left(\H^{2r} _{\dR} (\cx) / \F^{r} \H^{2r} _{\dR} (\cx) \right)^{\deg \Phi_i} \\
      v &\mapsto (\pi(v), \pi\circ q^{-r}\Frob_\k(v), \dots, \pi \circ (q^{-r} \Frob_\k)^{\deg \Phi_i -1}(v)),
    \end{aligned}
\end{equation}
from Section~\ref{sec:improve_obstruction}, is approximated by
\begin{equation}
[ \pi_{i, N} ] \colonequals [\pi_{\cx} ] \cdot
\left(
\begin{tabular}{c}
1 \\
$\phantom{{}^{\deg \Phi_i -1}}q^{-r}[\Frob_{\k,N}]\phantom{{}^{\deg \Phi_i -1}}$\\
$\vdots$\\
$\phantom{{}^{\deg \Phi_i -1}}(q^{-r} [\Frob_{\k,N}  ])^{\deg \Phi_i -1}$
\end{tabular}
\right)
\cdot B_{T_{i, N}}.
\end{equation}

\subsubsection{Bound dimension of $L_i$}\label{sec:dim Li}
The approximation $\pi_{i, N}$ of $\pi_{i}$ allows for the computation of a lower bound $b_{i,N}$ on the rank of $\pi_{i}$, see for example~\cite[Algorithm~1.62]{abbott-kedlaya-roe-10}.
Although we will not know when this happens, if $N$ is large enough, then the rank $b_{i,N}$ of $\pi_{i,N}$ will match the rank of $\pi_{i}$.
    Nonetheless, as $L_i \colonequals \ker \pi_i$, we have
    \begin{equation}
      \dim L_i = \dim T_i - \rk \pi_{i} \leq \dim T_i - b_{i,N}.
    \end{equation}

\subsubsection{Return the upper bound on the middle Picard rank}
From the previous argument at the end of the for loop we have $\sum_i \dim_{\Qxp} L_i \leq \texttt{bound}$.
Combining this with Proposition~\ref{prop:even_better_bound} we obtain the sought inequality:
\begin{equation}\label{eq:algo_bound}
  \dim_{\Qq} A^{r} (X_{\overline{K}}) \leq \sum_i \dim_{\Qxp} L_i \leq \texttt{bound}.
\end{equation}

\begin{remark}
  Let us warn once again that, even in the favorable conditions provided by K3 surfaces, the inequality~\eqref{eq:algo_bound} can be sharp.
  See Example~\ref{ex:weak} for a demonstration.
\end{remark}

\section{Examples}\label{sec:examples}

We now give explicit illustrations of the methods developed in this paper.
We have implemented a version of Algorithm~\ref{alg:main}, called \verb|crystalline_obstruction|\footnote{For its implementation in \texttt{SageMath}~\cite{sage}, see \url{https://github.com/edgarcosta/crystalline_obstruction}.},
where the prime is also given as input.
We will show its basic usage below.

There are three sets of examples below.
In the first set, we work with Jacobians of curves.
The Hodge structure on their cohomology is inherited from the Hodge structure of the corresponding curves.
The advantage is that the dimension of the cohomology is small enough that we can write the Frobenius matrices explicitly.

In the second set, we illustrate the basic usage on surfaces in projective space.
We checked all of the Picard numbers in the \emph{quartic database}\footnote{One may view the database at \url{https://pierre.lairez.fr/quarticdb/}.} of~\cite{lairez-sertoz}.
In every case, our upper bounds agreed with the numbers listed there. We also comment on the performance gain in using the obstruction method.

In the third set, we give pathological examples. For example, the obstruction space associated to a quintic surface is four dimensional. We give an example where the image of the obstructed Tate classes span a one dimensional space although four dimensions of algebraic classes must be obstructed.

Our convention for writing a $p$-adic number modulo $p^N$ is as follows: given $a \in \Qp$ we write $a \equiv p^m \cdot b \bmod p^N$ where $m \in \Zz$, $b\in \Zz_{\ge 0}$, $p \nmid b$.

\subsection{Jacobians of plane curves}\label{sec:jacobian}

The discussion on hypersurfaces above allows us to compute the Hodge decomposition on the first cohomology of a smooth curve.  Since the cohomology of the Jacobian is isomorphic to the natural Hodge structure on the wedge powers of the first cohomology of the curve, we can treat Jacobians explicitly.  We can thus use Algorithm~\ref{alg:main} with minor modifications, as illustrated below.  The purpose of this change in context is to be able to display complete examples in print.

\begin{example}\label{ex:genus2 trivial} We start with a genus $2$ curve
  \begin{equation}
    C/\Qq \colon y^2 - (4x^5 - 36x^4 + 56x^3 - 76x^2 + 44x - 23) = 0.
  \end{equation}
  Choose the prime $\xp = (31)$ and write $f$ for the equation of $C$.
  This curve has LMFDB label \href{https://www.lmfdb.org/Genus2Curve/Q/1104/a/17664/1}{\textsf{1104.a.17664.1}}. Let $J$ denote the Jacobian of $C$. We will compute the geometric Picard number of $J$.

When given a curve, the code \texttt{crystalline\_obstruction} understands that the intention is to compute with the Jacobian of the curve. The command is simple:
\begin{verbatim}
sage: from crystalline_obstruction import crystalline_obstruction
sage: crystalline_obstruction(f, p=31, precision=3)
\end{verbatim}
The first entry of the output is the upper bound (in this case sharp) on the geometric Picard number of $J$, and the second is a dictionary recording relevant intermediate results.  The entire computation for this example takes less than a second giving the output:
\begin{verbatim}
(1, {'precision': 3, 'p': 31, 'rank T(X_Fpbar)': 2, 'factors': [(t - 1, 2)],
     'dim Ti': [2], 'dim Li': [1]})
\end{verbatim}

We will now walk through the intermediate steps of the computation. We first check that $J$ has good reduction over $\frakp$.
Let $\cc$ and $\cj$ denote the natural models of $C$ and $J$ over $\Zxp$, respectively.
As discussed in Section~\ref{sec:computing_in_crystalline}, we may compute an approximation of $\Frob_\k$ on $\H^{1} _\dR(\cc/\Zxp)\otimes_{\Zxp} \Qxp \simeq \H^{1} _\dR(\cj/\Zxp) \otimes_{\Zxp} \Qxp$ and its characteristic polynomial:

\begin{equation}
  \begin{aligned}
    \Frob|_{\H^{1} _\dR(\cc/\Qxp)} \equiv&
\left(\begin{array}{rrrr}
31 \cdot 482  & 31 \cdot 284  & 16241  & 3075  \\
31 \cdot 386  & 31 \cdot 886  & 2644  & 12126  \\
31 \cdot 284  & 31 \cdot 659  & 6336  & 9750  \\
31 \cdot 194  & 31 \cdot 876  & 27408  & 10841
\end{array}\right)
\pmod{31^3},
\\
    \chi =& 1 - 3t + 14t^{2} - 93t^{3} + 961t^{4}.
\end{aligned}
\end{equation}
The natural Hodge structures on both sides of the isomorphism $\H^{2} _\dR(\cj/\Qxp) \simeq \Lambda^2 \H^{1} _\dR(\cj/\Qxp)$ agree with one another. Thus the wedge products of the Griffiths basis we used for $\cc$ will reveal the Hodge structure on higher cohomologies of $\cj$.

From $\chi$ and the approximate Frobenius above we can deduce
\begin{equation}
  \begin{gathered}
    \det(t - \Frob|_{\H^{2} _\crys(\cj_\frakp)(1)}) = (t - 1)^2 (31t^{4} + 48t^{3} + 43t^{2} + 48t + 31)/31,
    \\
  \Frob|_{\H^{2} _\dR(\cj/\Qxp)} \equiv
\left(\begin{array}{rrrrrr}
31^{2} \cdot 19  & 31 \cdot 660  & 31 \cdot 776  & 31 \cdot 843  & 31 \cdot 506  & 22499  \\
31^{2} \cdot 18  & 31 \cdot 250  & 31 \cdot 459  & 31 \cdot 270  & 31 \cdot 683  & 10699  \\
31^{2} \cdot 3  & 31 \cdot 154  & 31 \cdot 636  & 31 \cdot 261  & 31^{2} \cdot 24  & 3010  \\
31^{2} \cdot 22  & 31 \cdot 557  & 31 \cdot 664  & 31 \cdot 392  & 31^{2} \cdot 23  & 10438  \\
31^{2} \cdot 30  & 31 \cdot 77  & 31 \cdot 516  & 31^{2} \cdot 26  & 31 \cdot 449  & 3650  \\
31^{2} \cdot 7  & 31 \cdot 668  & 31 \cdot 509  & 31 \cdot 277  & 31 \cdot 513  & 17591
\end{array}\right)
 \pmod{31^3},
\end{gathered}
\end{equation}
and
    $T^1(\cj_\xp) = \Span(v_1, v_2)$,
where, having lost a digit of precision, we have
\begin{equation}
  \begin{aligned}
  v_1 \equiv& \left(356,\,37,\,831,\,0,\,295,\,31\right) \pmod{31^2}
  \\
  v_2  \equiv& \left(4,\,957,\,3,\,1,\,0,\,0\right) \pmod{31^2}.
  \end{aligned}
\end{equation}
The Hodge numbers for $\H^{2} _\dR(\cj/\Qxp)$ are $(1, 4, 1)$, and the projection
\begin{equation}
  \pi_{\cj} \colon \H^{2} _\dR(\cj/\Qxp) \to \H^{2} _\dR(\cj/\Qxp) / \F^1 \H^{2} _\dR(\cj/\Qxp) \simeq \Qxp
\end{equation}
corresponds to the projection onto the last coordinate in the basis chosen above. This gives
\begin{equation}
  \pi_T \equiv
  \left(
  \begin{array}{cc}
    31 & 0
  \end{array}
  \right) \pmod{31^2}
\end{equation}
in the basis $v_1,v_2$ for $T^1(\cj_\xp)$.  Thus, the rank of the obstruction map $\pi_T$ is at least $1$. This is the first entry in the output of our code.

We also know that the polarization lifts and thus $\dim_\Qq A^1(J_{\overline{\Qq}})_\Qq = \rk \End(J_{\overline{\Qq}}) = 1$.
\end{example}

\begin{example}
  Let us now look at an example where the rank of the Jacobian is known to be $2$.
  In this case, an upper bound of $2$ will not prove that the rank is indeed $2$. However, our numerical methods has the advantage that \emph{if} the rank was $1$ then the obstruction map would be non-zero at high enough precision. Observing that the obstruction is zero to higher and higher precision would be compelling evidence that the rank is $2$.

  We pick the genus $2$ curve $C \colon y^2 = x^{5} - 2x^{4} + 7x^{3} - 5x^{2} + 8x + 3$ over $\Qq$ with LMFDB label \href{https://www.lmfdb.org/Genus2Curve/Q/30976/a/495616/1}{\textsf{30976.a.495616.1}}.
  Mutatis mutandis, we may repeat Example~\ref{ex:genus2 trivial}, for $\xp \neq (2), (11)$. However, the Jacobian $J$ of $C$ will now have real multiplication by $\Qq(\sqrt{17})$, see~\cite{cmsv}.
  Therefore, $\dim_{\Qq} A^1(J_{\overline{\Qq}}) =2$.

Working with finite precision, it is impossible conclude that the obstruction map is identically zero.
We observe however that even over a large prime such as $p = 4999$ the obstruction map on the two dimensional $T^1(\cj/\Qxp)$ is zero modulo $p^{100}$. Even with these large numbers, the computation took about 3 minutes.
\end{example}

\begin{example} In this example, we will show how one can obstruct more than one linearly independent cycle at a single prime. Unlike~\cite{elsenhans-jahnel-11b}, this method does not require an investigation of the geometry of algebraic cycles.

  Consider the following genus 3 plane curve, its defining equation will be denoted by $f$:
  \begin{equation}
  C/\Qq \colon x y^{3} + x^{3} z - x y^{2} z + x^{2} z^{2} + y^{2} z^{2} - y z^{3} = 0.
  \end{equation}

  The following command takes less than a second to return the answer, which shows that $\dim A^1(\overline{J})=1$ where $J/\Qq$ is the Jacobian of $C$:
\begin{verbatim}
sage: crystalline_obstruction(f, p=31, precision=3)
(1, {'precision': 3, 'p': 31, 'rank T(X_Fpbar)': 3, 'factors': [(t - 1, 3)],
     'dim Ti': [3], 'dim Li': [1]})
\end{verbatim}
With $\xp=(31)$ and $\cc/\Zxp$, $\cj/\Zxp$ as in Example~\ref{ex:genus2 trivial}, we demonstrate the intermediate steps:
\begin{equation*}
  \begin{aligned}
    \chi =&  1  + 78t^{2} +  408t^{3} + 2418t^{4} + 7688t^{5} + 29791t^{6}
    \\
    \Frob|_{\H^{1} _\dR(\cc/\Qxp)} =&
\left(\begin{array}{rrrrrr}
31 \cdot 104  & 31 \cdot 218  & 31 \cdot 7  & 27783  & 2569  & 7195  \\
31 \cdot 351  & 31 \cdot 494  & 31 \cdot 690  & 19524  & 8323  & 1421  \\
31 \cdot 50  & 31 \cdot 237  & 31 \cdot 829  & 13467  & 20050  & 19610  \\
31 \cdot 19  & 31 \cdot 733  & 31 \cdot 377  & 20592  & 23805  & 15085  \\
31 \cdot 482  & 31 \cdot 610  & 31 \cdot 793  & 12397  & 28951  & 6604  \\
31 \cdot 710  & 31 \cdot 860  & 31 \cdot 689  & 19294  & 18382  & 25376
\end{array}\right)
\pmod{31^3}.
  \end{aligned}
\end{equation*}
Now the Hodge numbers for $\H^{2} _\dR(\cj/\Qxp)$ are $(3, 9, 3)$ and we have
  \begin{multline}
    \det(t -  \Frob|_{\H^{2} _\crys(\cj_\frakp)(1)}) = \frac{1}{31^3} (t - 1)^3
    \bigl(31^3
    + 15 \cdot 31^2 t
    -642 \cdot 31 t^2
  \\
    + 14545 t^3
    + 36060  t^{4}
    + 18063  t^5
    + 3172 t^6
    + 18063  t^7
    + \cdots\bigr).
  \end{multline}
  We find that the space of Tate classes, $T^1(\cj/\Qxp)=\ker \left( \Frob_\k - 1 \right)$, is three dimensional. The obstruction map $\pi_T$ is approximated by the following matrix with respect to a basis:
\begin{equation}
\pi_T \equiv \left(\begin{array}{rrr}
31 \cdot 240  & 0 & 31 \cdot 1  \\
31 \cdot 515  & 31 \cdot 1  & 0 \\
0 & 0 & 0
\end{array}\right)
\pmod{31^3}.
\end{equation}
Thus, $\rk \pi_T \geq 2$ and we conclude $\dim A^1(J_{\Qbar}) = \dim \End(J_{\Qbar}) = 1$.
\end{example}

\begin{example}
  We may also use our method to directly bound the dimension of the geometric endomorphism algebra of an abelian variety.
  We do this via divisorial correspondences.

  If $X$ and $Y$ are two projective smooth varieties over a number field $K$, then
  \begin{equation}
    \begin{aligned}
      \NS(X \times Y) &= \NS(X) \oplus \NS(Y) \oplus \DC(X,Y)\\
                      &\subset \H^2 _\dR(X/K) \oplus \H^2 _\dR(Y/K) \oplus \H^1 _\dR(X/K) \otimes \H^1 _\dR(Y/K),
  \end{aligned}
  \end{equation}
  where $\NS$ denotes the Néron--Severi group and $\DC$ denotes divisorial correspondences. The Néron--Severi group $\NS$ coincides with our $A^1$.

  Furthermore, by \cite[VI \S 2 Thm 2]{lang-1983} we have the following isomorphism
  \begin{equation}
    \DC(X, Y) \simeq \operatorname{Hom}(\Alb(X), \xP(Y)),
  \end{equation}
  where $\Alb(X)$ is the Albanese variety of $X$ and $\xP(Y)$ is the Picard variety of $Y$.

  Note that the Hodge structure of  the tensor $\H^1 _\dR(\cX/\Qp) \otimes \H^1 _\dR(\cy/\Qp)$ is induced by the Hodge structure of its two factors. Therefore, knowing the Hodge structure on the factors allows us to apply our method to bound the dimension of $\DC(X_{\Qbar}, Y_{\Qbar})$.

  For example, taking $X=Y=C$ for a curve $C/K$, we can bound the dimension of the endomorphism algebra $\End(J_{\Qbar})$ of the Jacobian $J$ of $C$. The method is analogous to the previous examples. We let $\cj/\Zxp$ denote the natural model of $J$ and replace $\H^{2} _\dR(\cj/\Qxp)$ with $\H^{1} _\dR(\cj/\Qxp) \otimes \H^{1} _\dR(\cj/\Qxp)$ in the computations. 

  Our code automates this procedure with the keyword argument \verb|tensor=True|. For example, take the smooth plane curve
  \begin{equation}
    C/\Qq \colon -y^{4} + x^{3} z + 2 x^{2} z^{2} - x z^{3} = 0.
  \end{equation}
  The following commands show that $\dim \End(J_{\Qbar}) \leq 9$ and $\dim \NS(J_{\Qbar}) \leq 5$, respectively.
\begin{verbatim}
sage:crystalline_obstruction(f, p=43, precision=4,tensor=True)
(9, {'precision': 4, 'p': 43, 'rank T(X_Fpbar)': 18,
  'factors': [(t + 1, 4), (t - 1, 6), (t^2 + 1, 4)],
  'dim Ti': [4, 6, 8], 'dim Li': [2, 3, 4]})
sage:crystalline_obstruction(f, p=43, precision=4)
(5, {'precision': 4, 'p': 43, 'rank T(X_Fpbar)': 7,
  'factors': [(t + 1, 2), (t - 1, 3), (t^2 + 1, 1)],
  'dim Ti': [2, 3, 2], 'dim Li': [1, 2, 2]})
\end{verbatim}

\end{example}

\subsection{Surfaces in projective space}

Quartic surfaces and quintic surfaces are the prime examples in this section. We present some successful examples and some examples that demonstrate inherent weaknesses of the method.

\subsubsection{K3 surfaces}

We begin with quartic smooth surfaces in $\ppp$.
The middle cohomology of a K3 surface has dimension 22 with Hodge numbers~$(1, 20, 1)$.

\begin{example}\label{ex:database}
  We applied our Algorithm~\ref{alg:main} to each of the \numprint{184725} quartic surfaces in the \emph{quartic database}\footnote{One may explore the database at \url{https://pierre.lairez.fr/quarticdb/}.}~\cite{lairez-sertoz}. For each quartic, we found an upper bound that agreed with the Picard number listed there (and, as expected, could not find a smaller upper bound).

We will now compare the performance improvement of using the obstruction map. We distinguish between three different methods: The \emph{Galois obstruction method} applies Algorithm~\ref{alg:main} thereby using the obstruction map on each of the Galois orbits separately as in~Proposition~\ref{prop:even_better_bound}. The \emph{vanilla obstruction method} deviates from Algorithm~\ref{alg:main} only in ignoring the Galois structure on the space of Tate classes, that is, the obstruction map is considered on the entire space of Tate classes.
For each surface, we start with the prime $23$ and move up to the next prime if the upper bound does not match the number listed in the database or if the prime is not a good prime.
We skip the primes less than $23$, as the computation to deduce the Hasse--Weil zeta for these requires more precision~\cite[Section~1.6.2]{costa-phd}, and hence more time.

In applying the \emph{van~Luijk method}~\cite{vanluijk-07} to each surface, it would make sense to stop after \emph{two primes} have been found where the upper bound is optimal, whereby the discriminants can be compared. This means it would require, on average, twice as many prime reductions as the vanilla obstruction method. Almost the entirety of the computation per prime is spent on computing the zeta function, a feature common to all three methods. Therefore, the van~Luijk method would take roughly twice as long as the vanilla obstruction method. On Figure~\ref{fig:number_of_attempts}, its slope would be half that of the vanilla obstruction --- or about one fifth of the Galois obstruction method.

  The entire computation for the Galois obstruction method took about 10 months of CPU time.
  The vanilla obstruction method takes about 16 months. Arguing as above, we expect the van~Luijk method to take about 32 months. The Figures~\ref{fig:number_of_attempts} and~\ref{fig:time_vs_percentile} give more detailed information.

\begin{figure}
  \centering
  \includegraphics[width=0.9\linewidth]{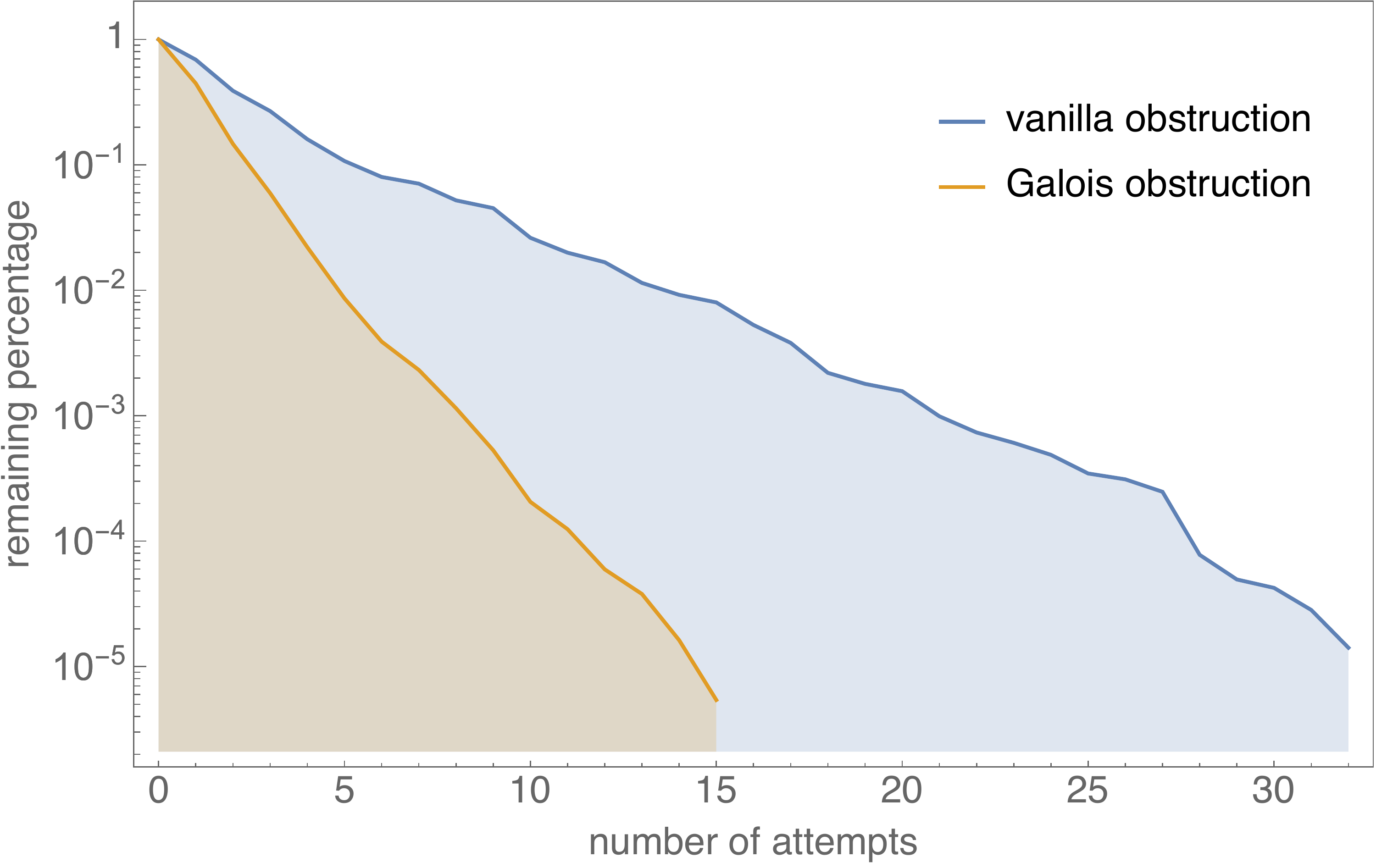}
  \caption{The percentile of surfaces from the quartic database that needed a certain number of prime reductions until sharp upper bound is attained.}\label{fig:number_of_attempts}
\end{figure}

\begin{figure}
  \centering
  \includegraphics[width=0.9\linewidth]{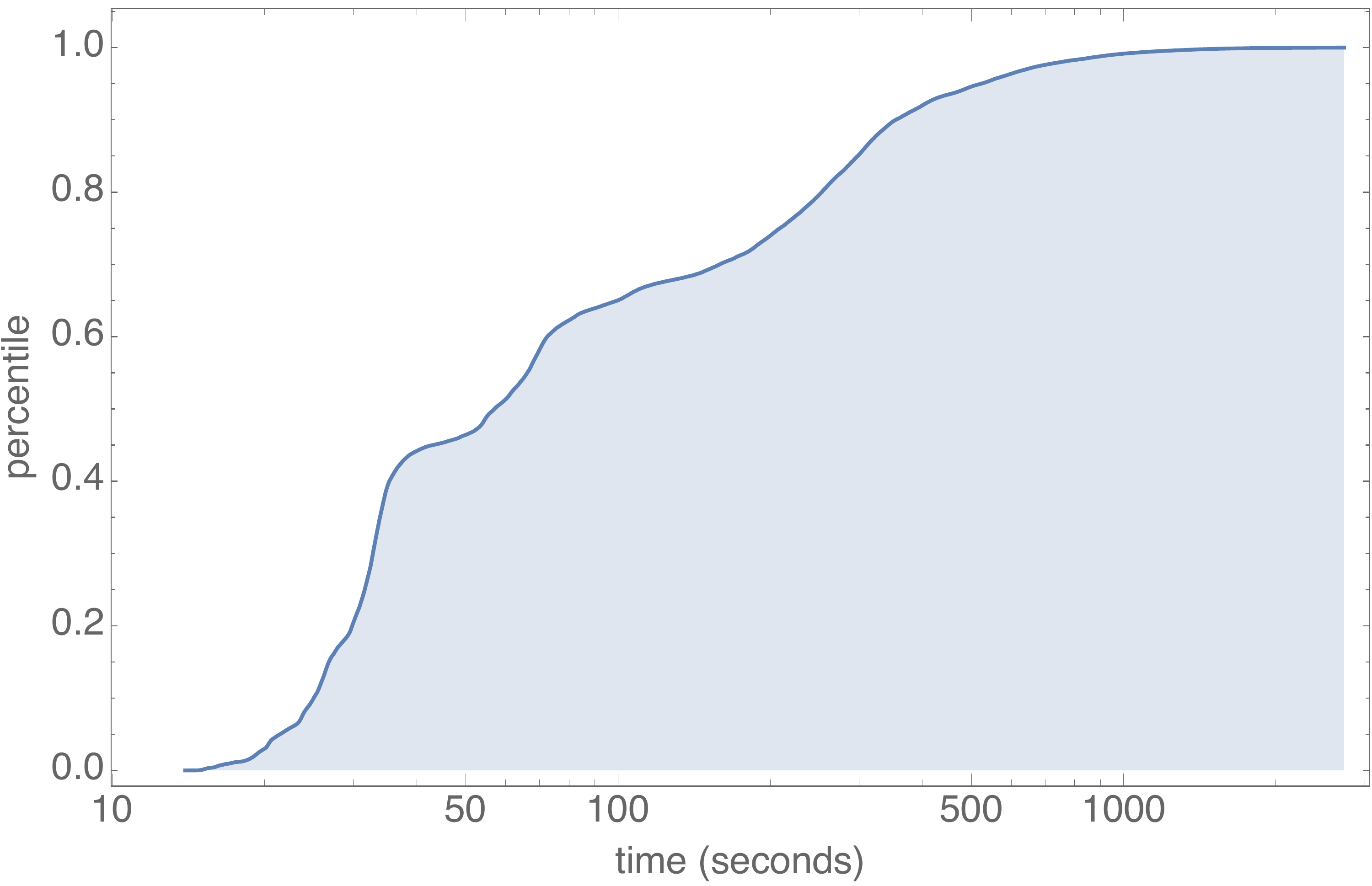}
  \caption{Percentile of surfaces from the quartic database against time needed until a sharp bound is attained via Galois obstruction method.}\label{fig:time_vs_percentile}
\end{figure}
\end{example}

\begin{example} In this example, we give a simple but artificial example which demonstrates that there is no \emph{global} upper bound on how much precision must be used to reach the optimal bounds of the method.

Fix a prime $p$ and let $\lambda= p^m$ with large $m\gg0$. Consider the following two elements of the Dwork pencil
  \begin{eqnarray}
    X/\Qq \colon& x^4 + y^4 + z^4 + w^4 + \lambda x y z w = 0, \\
    Y/\Qq \colon& x^4 + y^4 + z^4 + w^4 = 0.
  \end{eqnarray}
  The Fermat quartic $Y$ has Picard number 20 and it is known that $Y_p$ has Picard number 20 when $p \equiv 1 \bmod 4$. Clearly, the two models $X$ and $Y$ lifting $X_p=Y_p$ are indistinguishable modulo $p^m$. If $X$ has Picard number 19, then the obstruction map will be zero modulo $p^m$ as it will not be able to make the distinction between $X$ and $Y$.
 Although, when the $p$-adic precision is sufficiently large, the obstruction becomes non-zero.
\end{example}

\begin{example}\label{ex:galois}
  Take $\frakp = (89)$ and consider the following quartic surface
  \begin{equation}\label{eqn:K3_galois}
    X/\Qq \colon y^{4} - x^{3} z + y z^{3} + z w^{3} + w^{4} = 0.
  \end{equation}
  The characteristic polynomial of $89^{-1} \Frob_\xp$ on $\H^2 _{\crys}(X_\frakp / \Qxp)$ factors as
  \begin{equation}
   (t - 1)^{5} \cdot (t + 1) \cdot (t^{4} + 1) \cdot (89  + 188 t^{2} + 303 t^{4} + 316 t^{6} + 303 t^{8} + 188 t^{10} + 89 t^{12})/89.
  \end{equation}
  Therefore, the geometric Picard number of $X$ is at most $10$.  By computing an approximation of the obstruction map on the Tate classes, we would at most improve this bound by $1$. Instead, we use Corollary~\ref{cor:better_bound} to drop the rank from $10$ to $4$.

  For instance, we observe that the obstruction is non-zero on the space
  \begin{equation}
    \Ttilde^r _8 (X_\frakp) \colonequals \ker( \Frob_k ^4 + p^4 ) \subset \H^{2} _{\crys} (X_\xp/\Zz_\xp).
  \end{equation}
  This already allows us to drop from $10$ to $6$, since $\deg(t^4 + 1) = 4$. Analogously, we work through the other two cyclotomic factors, observing a non-zero obstruction in each case. As these two cyclotomic polynomials are linear, we drop from $6$ to $4$.

  All of this is automated. One may use \verb|crystalline_obstruction| to deduce the bound~$4$. The following computation takes less than 4 minutes of CPU time.
\begin{verbatim}
sage: crystalline_obstruction(f, p=89, precision=3)
(4,
 {'precision': 3, 'p': 89, 'rank T(X_Fpbar)': 10,
  'factors': [(t - 1, 1), (t + 1, 1), (t - 1, 4), (t^4 + 1, 1)],
  'dim Ti': [1, 1, 4, 4], 'dim Li': [1, 0, 3, 0]})
\end{verbatim}
\end{example}

\begin{example}\label{ex:weak}
  Consider now the K3 surface
  \begin{equation}
X/\Qq \colon x^{4} + 2 y^{4} + 2 y z^{3} + 3 z^{4} - 2 x^{3} w - 2 y w^{3} = 0.
  \end{equation}
  At $\frakp = (43)$ the geometric Picard number of $X_\frakp$ is two, and one Tate class is obstructed, thus the geometric Picard number of $X$ is one.

  Now take $\frakp = (67)$.
  The characteristic polynomial of $67^{-1} \Frob_\frakp$ factors as
  \begin{multline}
    (t - 1) \cdot (t + 1)^3 \cdot (67  - 171 t^{1} + 181 t^{2} - 53 t^{3} - 71 t^{4} + 9 t^{5}\\+ 178 t^{6} - 215 t^{7} + 54 t^{8} + 50 t^{9} + 54 t^{10}
    \ldots)/67,
  \end{multline}
  thus the geometric Picard number of $X_\frakp$ is $4$.

  As the polarization is unobstructed, there is only one irreducible space to consider, corresponding to $(t + 1)^3$.
  Therefore, in contrast with the previous example,  we may only obstruct the liftability of one class, as the obstruction space is one dimensional.
  Hence, we cannot obtain a sharp upper bound for the geometric Picard number of $X$ by studying its specialization at $\frakp = (67)$.
\end{example}

\begin{example}\label{ex:real_multiplication}
  Now we look at a K3 surface with real multiplication.
  Consider the K3 surface given by desingularizing the following double cover of the projective plane.
  It is given in the weighted projective plane $\Pp(1,1,3)$ and is branched over six lines.
  \begin{equation}
    X/\Qq \colon w^2 = \left(- y^{2}/8 + y z - z^{2}\right) \left(7x^{2}/8 + 5 x z + 7 z^{2} \right) \left(2 x^{2} + 3 x y + y^{2}\right).
  \end{equation}
  This is the specialization of the family $X^{(2,t)}$ in~\cite[Theorem~6.6]{ej-rm} with $t = 1$. By \emph{loc.\ cit.}\ it has geometric Picard number 16 and real multiplication by $\Qq(\sqrt{2})$.
  In particular, every reduction of $X$ will overestimate its geometric Picard rank by a multiple of $2$~\cite{charles-14}. In fact, the reductions must have geometric Picard rank 18 or 22~\cite[Corollary~4.12]{ej-rm}.
  The van Luijk~\cite{vanluijk-07} method fails in this case. We now show that the obstruction method will succeed.

  Take the prime $\frakp = (83)$. By counting points on the singular model, we deduce that the characteristic polynomial of $83^{-1} \Frob_\frakp$ can be factored as $\chi_{1} \cdot \chi_{2}$ where
  \begin{equation}
    \chi_{1} = (t - 1)^{10} (t + 1)^6,\quad \chi_{2} = (t^2 + 1) (83 - 158 t^2 + 83t^4) /83.
  \end{equation}
  Here, $\chi_{1}$ corresponds to the action of the Frobenius on the known $16$ algebraic cycles on $X$. The factor $\chi_2$ corresponds to the action on their orthogonal complement (the transcendental part).

  In keeping with the notation of Proposition~\ref{prop:cylic_splitting}, let $T_4 ^1(X_{\xp}) \colonequals \ker\left( \Frob^2 + p^2 \right) $. From the argument above and Proposition~\ref{prop:decompositions_agree}, we conclude that $T_4^1(X_\xp)$ spans algebraic cycles in $\overline{X}_\xp$ which do not lift. In particular, by Theorem~\ref{thm:simplified bloch-esnault-kerz-14}, the obstruction map must be non-zero on $T_4^1(X_\xp)$.

  Consequently, using sufficiently high precision, one can detect that the obstruction map $\pi_{T_4}$ is non-zero. Using Corollary~\ref{cor:better_bound} one would then conclude that the geometric Picard number of $X$ is $16$.
  We did not implement the computation of the Frobenius on a smooth variety given as the resolution of a singular hypersurface.
  Therefore, we cannot comment on the precision required.
\end{example}

\subsubsection{Quintic surfaces}

We now consider smooth quintic surfaces.
For these surfaces the Hodge numbers are $(4, 45, 4)$.
Now the obstruction space is four dimensional, and thus favoring richer examples.

\begin{example}
    Consider the following quintic surface
    \begin{equation}
      X/\Qq \colon 10 x^{4} y - 3 z^{5} + 3 x^{4} w + 3 y^{4} w - 10 x w^{4} - 23 x^{3} y z = 0.
    \end{equation}
    We compute that $X_\frakp$ has Picard number one at the prime $\frakp=(13)$.
    This means the Picard number of $X$ is one.
    On the other hand, the reduction $X_\frakp$ at $\frakp=(23)$ has $5$-dimensional space of Tate classes.
    Fortunately, the obstruction space is $4$-dimensional and we see that $\pi_T$ is indeed rank $4$.
    Thus, this method allows for the correct determination of the Picard rank of $X$ at a prime where the jump in Picard number is four.
    It took about 1 day of CPU time per prime to compute the Hasse--Weil zeta function and an approximation for the Frobenius matrix.
\end{example}

\begin{example}\label{ex:drop_by_one}
The following example suggests that we are not necessarily immune to symmetries in the transcendental lattice.
  Take the quintic surface
  \begin{equation}
    X/\Qq \colon 3x^4z + 9xy^4 + 9y^2z^3 + z^5 + 5w^5 = 0.
  \end{equation}
  The geometric Picard number of $X_\frakp$ at $\frakp=(31)$, and therefore of $X$, is one.
  On the other hand, the dimension of the space of eventual Tate classes $\Ttilde^1(X_\frakp)$ of $X_\frakp$ at $\frakp=(29)$ is $5$, and the characteristic polynomial of the Frobenius action on $\Ttilde^1(X_\frakp)$ is $(t - 31)^3 (t + 31)^2.$
  This time, the following result shows that there is another theoretical limit to the best possible upper bound we can compute at $\frakp=(29)$.

  We recall the decomposition $\Ttilde^1(X_\frakp) = T_1^1(X_\xp) \bigoplus T_2^1(X_\xp)$ as in Proposition~\ref{prop:cylic_splitting}.

  \begin{proposition}
    The images of the obstruction maps $\pi_{T_1}$ and $\pi_{T_2}$ coincide. This image has dimension~$1$, although the codomain has dimension~$4$.
  \end{proposition}

  \begin{proof}
    We use the action of the involution $\iota\colon y \mapsto -y$ on the obstruction space and on the Tate classes. By inspecting the Griffiths basis, it is immediately seen that the $(-1)$-eigenspace of $\iota^*$ on the obstruction space is one dimensional.

    We claim that all primitive Tate classes are $(-1)$-eigenvectors of $\iota^*$ on cohomology. To show this, we computed an approximation of $\Ttilde^1(X_\frakp)$ and projected this approximation to the $(-1)$-eigenspace of $\iota^*$. The minimal valuation of the $4 \times 4$ minors of a basis of the image is $0$, even with $10$ digits of approximation. This proves the claim.

    The obstruction map is $\iota^*$ equivariant and the polarization is unobstructed. Therefore, the image of the Tate classes via $\pi_{\Ttilde}$ must be confined to a one dimensional space. We observe that the two ranks are at least $1$ by computation.
  \end{proof}

  We note another behavior to this problem at another prime. At $\frakp = (23)$, the characteristic polynomial of the Frobenius action on $\Ttilde(X_\frakp)$ is now $(t - 23)^2 (t + 23) (t^2 + 23^2)$. Since the Galois representation $\Ttilde^1(X_\xp)$ splits into more invariant pieces, and one piece is two dimensional, we can drop the rank to $1$ using Corollary~\ref{cor:better_bound}.
  It takes about 1 CPU day per prime to get these results.

  \begin{remark}
    By the computations at $\frakp = (23), (29)$ and $(31)$, we know that the dimension of the endomorphism algebra of the transcendental lattice is at most 4.
    Furthermore, we note that $X$ has another automorphism given $w \mapsto \zeta_5 w$, thus the transcendental lattice has CM by $\Qq(\zeta_5)$.
    Then, one expects the Picard number of the reductions to be congruent to 1 modulo 4~\cite{charles-14}.
  \end{remark}
\end{example}

\printbibliography
\end{document}